\documentclass[11pt]{amsart}
\usepackage[utf8]{inputenc}

\usepackage[nospace,noadjust]{cite}
\usepackage{relsize}
\usepackage{cite}
\usepackage{color}
\usepackage[dvipsnames]{xcolor}

\usepackage{enumitem}

\usepackage{tikz-cd}
\usetikzlibrary{calc,fit,matrix,arrows,automata,positioning}
\usetikzlibrary{decorations.pathreplacing,angles,quotes}
\usepackage{colortbl}
\usepackage{hyperref, enumitem}
\hypersetup{
  colorlinks   = true, 
  urlcolor     = blue, 
  linkcolor    = Purple, 
  citecolor   = red 
}
\usepackage{amsmath,amsthm,amssymb}

\usepackage{pgfplots}
\pgfplotsset{compat=1.18}
\usepackage{xfp} 
\usepackage{tikz}
\usetikzlibrary{math} 

\usepackage[margin=2.5cm]{geometry}

\usepackage{mathtools}
\DeclarePairedDelimiter\ceil{\lceil}{\rceil}
\DeclarePairedDelimiter\floor{\lfloor}{\rfloor}

\usepackage{stmaryrd}
\usepackage{comment}

\usepackage{array}
\newcolumntype{x}[1]{>{\centering\arraybackslash\hspace{0pt}}p{#1}}

\theoremstyle{definition}
\newtheorem{theorem}{Theorem}[section]
\newtheorem{definition}[theorem]{{{Definition}}}
\newtheorem{example}[theorem]{{{Example}}}

\newtheorem{remark}[theorem]{{{Remark}}}

\newtheorem{corollary}[theorem]{{{Corollary}}}
\newtheorem{proposition}[theorem]{{{Proposition}}}
\newtheorem{lemma}[theorem]{{{Lemma}}}

\newtheorem{construction}{{{Construction}}}

\newcommand{\numberset}{\mathbb}

\newcommand{\Z}{\numberset{Z}}

\newcommand{\C}{\mathcal{C}}

\newcommand{\F}{\numberset{F}}

\newcommand{\wt}{\textnormal{wt}}

\newcommand{\Fq}{\F_q}

\newcommand{\Fqm}{\mathbb{F}_{q^{m}}}
\newcommand{\Fqmn}{\mathbb{F}_{q^{m}}^{n}}

\newcommand{\cc}{\mathbf{c}}
\newcommand{\xx}{\mathbf{x}}
\newcommand{\vv}{\mathbf{v}}
\newcommand{\uu}{\mathbf{u}}
\newcommand{\w}{\textnormal{w}}

\usepackage{cleveref}

\newcommand{\rk}{\textnormal{rk}}

\DeclareMathOperator{\GL}{GL}

\DeclareMathOperator{\lf}{\lfloor}
\DeclareMathOperator{\rf}{\rfloor}

\newtheorem*{maintheorem}{Main Theorem}

\newcommand{\nkdqm}{[n,k,d]_{q^m/q}}
\newcommand{\nkqm}{[n,k]_{q^m/q}}

\title{The maximum number of nonzero weights of linear rank-metric codes}

\author[C. Castello]{Chiara Castello}
\address{Chiara Castello, \textnormal{Dipartimento di Matematica e Fisica,
Universit\`a degli Studi della Campania ``Luigi Vanvitelli'',
I--\,81100 Caserta, Italy.}}
\email{chiara.castello@unicampania.it}

\author[P. Santonastaso]{Paolo Santonastaso}
\address{Paolo Santonastaso, \textnormal{Dipartimento di Matematica e Fisica,
Universit\`a degli Studi della Campania ``Luigi Vanvitelli'',
I--\,81100 Caserta, Italy\newline
Dipartimento di Meccanica, Matematica e Management, 
Politecnico di Bari, 
70125 Bari, Italy.}}
\email{paolo.santonastaso@poliba.it}

\author[M. Scotti]{Martin Scotti}
\address{Martin Scotti, \textnormal{Aalborg University, \newline
Department of Mathematical Sciences, 
Thomas Manns Vej 23,
9220 Aalborg Øst, Denmark.}}
\email{msco@math.aau.dk}

\begin{document}

\begin{abstract}
We investigate the maximum number \( L_{\rk}(n, m, k, q) \) of distinct nonzero rank weights that an \( \F_{q^m} \)-linear rank-metric code of dimension \( k \) in \( \F_{q^m}^n \) can attain. We determine the exact value of the function \( L_{\rk}(n, m, k, q) \) for all admissible parameters \( n, m, k, q \). In particular, we characterize when a code achieves the \emph{full weight spectrum} (FWS), i.e. when the number of distinct nonzero rank weights equals \( \min\{n, m\} \). We provide both necessary and sufficient conditions for the existence of FWS codes, along with explicit constructions of codes attaining the maximum number of distinct weights. We discuss the equivalence of such codes and also present classification results for 2-dimensional codes. Finally, we investigate further properties of these optimal codes, like their behavior under duality. 
\end{abstract}

\maketitle


\section{Introduction}

The study of the \emph{weight spectrum} of codes has a rich and fruitful history, with investigations carried out from various perspectives and for multiple purposes over the years. In \cite{delsarte1973four}, Delsarte investigated the size of a code with a fixed number of distinct distances, which, in the case of linear codes, corresponds to analyzing the size of a code with a fixed number of distinct (nonzero) weights.  On the other hand, in \cite{macwilliams1963theorem}, MacWilliams considered the foundational problem of whether, given a set $S$ of positive integers, it is possible to construct a code whose set of nonzero weights coincides exactly with $S$, and provided partial answers to this question, along with necessary conditions on the set $S$.
In \cite{gorla2023integer}, this problem was recently revisited in the context of \emph{generalized weights} under different metrics. Other discussions on the structure and cardinality of the weight set appear in \cite{enomoto1987codes,slepian1956class}. 

For the Hamming metric case, Shi et al.\ in \cite{shi2019many} investigated the more general combinatorial question of determining the maximum number $L_H(k, q)$ of distinct nonzero Hamming weights that a linear code of dimension $k$ over the finite field $\mathbb{F}_q$ with $q$ elements can attain. They provided the upper bound
\begin{equation} \label{eq:hammingL(k,q)}
L_H(k, q) \leq \frac{q^k - 1}{q - 1}.
\end{equation}
This bound was shown to be tight in the binary case in \cite{haily2015binary}, and it was further proven in \cite{shi2019many} to be sharp for all $q$-ary linear codes of dimension $k = 2$. The authors then conjectured that the bound is tight for all values of $q$ and $k$. This conjecture was later proved by Alderson and Neri in \cite{alderson2018maximum}, where the existence of such codes was proven for all pairs $(k, q)$, provided that the length $n$ is sufficiently large. Linear codes achieving the bound in \eqref{eq:hammingL(k,q)} are referred to as \textbf{maximum weight spectrum (MWS) codes}. A natural refinement of the problem considers the combinatorial function $L_H(n, k, q)$, defined as the maximum number of distinct nonzero weights that a linear code can exhibit for a fixed length $n$. Since the Hamming weight of any codeword cannot exceed $n$, it follows trivially that \[
L_H(n, k, q) \leq n.\] In \cite{alderson2019note}, Alderson completely characterized the parameters for which this upper bound is attained, and such codes were called \textbf{full weight spectrum (FWS) codes}. However, when the length $n$ is fixed, determining the exact value of $L_H(n, k, q)$ seems challenging, and partial results have been obtained in \cite{shi2019many,alderson2019note}. These types of combinatorial questions have been further extended to linear codes endowed with general coordinate-wise weight functions, such as the Lee weight and the Manhattan-type weight \cite{alderson2025mws}. More recently, FWS codes have also been explored in the setting of \emph{subspace codes} \cite{castello2024full,shi2025new}.

In the same spirit as the above-mentioned combinatorial problems for the weight spectrum of codes, this paper introduces and investigates analogous questions in the context of rank-metric codes. 

In recent years, rank-metric codes have attracted significant attention due to their wide range of applications and their deep connections to rich mathematical structures. The concept of rank-metric codes was first introduced by Delsarte in 1978~\cite{delsarte1978bilinear}, and was later independently rediscovered by Gabidulin~\cite{gabidulin1985theory} and Roth~\cite{roth1991maximum}. Since then, these codes have found applications in different areas such as criss-cross error correction, cryptography, and network coding. Moreover, rank-metric codes are intimately connected with various well-established algebraic and combinatorial objects, including semifields, linear sets in finite geometry, tensorial and skew algebras, $q$-analogs of matroids, and several others. For a comprehensive treatment of both the practical and theoretical aspects of rank-metric codes, we refer the reader to the surveys~\cite{gorla2018codes,sheekeysurvey,bartz2022rank,gruica2023rank}.

In this paper, we consider linear rank-metric codes. Let $\mathbb{F}_{q^m}$ denote the extension of degree $m$ of the finite field $\mathbb{F}_q$. For a vector $\mathbf{v} = (v_1, \ldots, v_n) \in \mathbb{F}_{q^m}^n$, the \textbf{(rank) weight} $\mathrm{w}(\mathbf{v})$ is defined as the $\mathbb{F}_q$-dimension of the $\mathbb{F}_q$-vector space generated by its entries, i.e.
\[
\mathrm{w}(\mathbf{v}) := \dim_{\mathbb{F}_q} \langle v_1, \ldots, v_n \rangle_{\mathbb{F}_q}.
\]
A \textbf{(linear) rank-metric code} $\mathcal{C}$ is a $k$-dimensional $\mathbb{F}_{q^m}$-linear subspace of $\mathbb{F}_{q^m}^n$, equipped with the \textbf{rank distance} defined by
\[
d(\mathbf{u}, \mathbf{v}) := \mathrm{w}(\mathbf{u} - \mathbf{v}),
\]
for all $\mathbf{u}, \mathbf{v} \in \mathcal{C}$.

\subsection{Our Contribution} In this work, we investigate the maximum number of distinct nonzero weights that a linear rank-metric code can attain. Analogously to the Hamming metric case, we introduce and study the combinatorial function $L_{\mathrm{rk}}(n, m, k, q),$
which denotes the maximum number of distinct nonzero rank weights that a linear rank-metric code of dimension $k$ in $\mathbb{F}_{q^m}^n$ can have.
A first trivial upper bound is given by
\begin{equation} \label{eq:trivialrankbound}
L_{\mathrm{rk}}(n, m, k, q) \leq \min\{m, n\}.
\end{equation}
In analogy with the Hamming metric case, we refer to rank-metric codes attaining this bound as \textbf{full weight spectrum (FWS)} rank-metric codes. In other words, an FWS rank-metric code has the same set of weights as the ambient space $\mathbb{F}_{q^m}^n$. We determine necessary and sufficient conditions for the existence of such codes. These results can be seen as the rank-metric counterpart of those established in~\cite{alderson2019note} for the Hamming metric.

In addition, in the case where $L_{\mathrm{rk}}(n, m, k, q) < \min\{m, n\}$, we determine the exact value of the function for every admissible choice of parameters $n, m, k, q$, under the assumption that the code is \emph{nondegenerate}, meaning that it cannot be isometrically embedded into a smaller ambient space. In particular, this imposes the additional restriction $n \leq km$.
We divide our analysis into two cases, depending on the length $n$. First, we analyze the case \(n \le m\).
We begin by providing a lower bound on \(L_{\mathrm{rk}}(n,m,k,q)\) by explicitly
constructing codes whose weight spectrum has a given size. Since
\(L_{\mathrm{rk}}(n,m,k,q)\) is defined as the maximum possible size of the
rank-weight spectrum, the existence of such codes immediately implies that
\(L_{\mathrm{rk}}(n,m,k,q)\) is at least this value
(see Construction~\ref{constr:casenleqm} and Theorem~\ref{thm:constrnleqm}).
Then, using a combinatorial argument on the possible rank-weight spectra of linear
rank-metric codes, we show that this bound is also an upper bound on
\(L_{\mathrm{rk}}(n,m,k,q)\) when \(n \le m\)
(see Theorem~\ref{th:upperboundn<m}). As expected, in this regime the function grows linearly with $n$. We then turn to the case \(n > m\). As in the previous case, we first establish a lower bound on
\(L_{\mathrm{rk}}(n,m,k,q)\) by explicitly constructing codes whose weight
spectrum attains this value
(see Construction \ref{constr:casen>m} and Theorem \ref{thm:constrn>m}). Next, adopting a geometric perspective and exploiting the interpretation of
rank-metric codes in terms of \(q\)-systems introduced in
\cite{Randrianarisoa2020ageometric}, we prove that this bound is also an upper bound
on \(L_{\mathrm{rk}}(n,m,k,q)\) in the case \(n > m\)
(see Theorem \ref{th:boundcasen>m}). As a byproduct of our investigations, we will obtain, for every fixed choice of $k$, $m$, and $q$, the maximum amount of nonzero weights for the special case $m=n$. Interestingly, when $n > m$ and $n$ increases, the function $L_{\mathrm{rk}}(n, m, k, q)$ does not decrease uniformly. Instead, its behavior strongly depends on the ratio $n/m$. Notably, in contrast with the Hamming case, where the asymptotic behavior of $L_H(n, k, q)$ as $n \to \infty$ converges to $L(k, q)$ cf. \cite[Theorem 4]{shi2019many}, in the rank-metric setting, the function $L_{\mathrm{rk}}(n, m, k, q)$ exhibits different and more intricate behavior. In particular, as $n$ approaches $km$, the maximum number of distinct weights tends to 1. In this extremal case, the code corresponds to the simplex code in the rank metric \cite{Randrianarisoa2020ageometric,alfarano2022linear}. Also, while the exact values of $L_H(n, k, q)$ remain unknown for different values of $n$ in the Hamming case, in this paper, for the rank-metric case, we are able to completely determine the values of $L_{\mathrm{rk}}(n, m, k, q)$ for all admissible parameters, along with explicit constructions of codes attaining them.

\smallskip
Our main result is the following. 

\begin{maintheorem} \label{thm:maintheorem}
For all $n, m, k, q$ with $n \leq km$, the function $L_{\mathrm{rk}}(n, m, k, q)$ satisfies
$$
L_{\mathrm{rk}}(n, m, k, q) =
\begin{cases}
n - \max\left\{ \left\lfloor \dfrac{n}{2^{k-1}} \right\rfloor, 1 \right\} + 1 & \text{if } n \leq m, \\[1em]
m - \max\left\{ \left\lfloor \dfrac{n - (k - t - 2)m}{2^{t+1}} \right\rfloor, 1 \right\} + 1 & \text{if } n > m,
\end{cases}
$$
where $t = \left\lfloor k - \dfrac{n}{m} \right\rfloor$.
\end{maintheorem}

We also investigate the equivalence problem for such codes, showing that there exist non-equivalent rank-metric codes achieving the maximum number of distinct weights. Moreover, by establishing a connection with the classification result for 
\emph{full weight spectrum one-orbit cyclic subspace codes} given in 
\cite[Theorem 1.2]{castello2024full}, we provide a classification result for 
2-dimensional codes, when $n \leq m$, whose weight spectrum have size 
$L_{\mathrm{rk}}(n, m, 2, q)$.
Finally, we discuss various properties of codes attaining the maximum number of 
distinct weights, showing that this class of codes does not fall into the class of 
MRD codes (unless \(k = 1\)), and we investigate their behavior under duality.

\subsection{Organization of the paper}The paper is organized as follows. In Section \ref{sec:preliminaries}, we introduce the general framework and preliminary notions required to study the analogue of the maximum weight spectrum problem in the context of rank-metric codes. Sections \ref{sec:nleqm} and \ref{sec:n>m} are devoted to determining the exact value of the function $L_{\mathrm{rk}}(n, m, k, q)$ in the cases $n \leq m$ and $n > m$, respectively. In Section \ref{section:nonequivalence}, we discuss the equivalence issue for codes 
having the maximum number of distinct weights, and we provide a classification result for 2-dimensional codes with 
\(n \le m\). Finally, Section \ref{sec:conclusions} provides concluding remarks and presents open problems for future research.
In \Cref{appendix}, we explore futher properties of codes attaining the maximum number of distinct weights.

\section{The Setting}
\label{sec:preliminaries}

We consider \emph{linear rank-metric codes}, i.e. $k$-dimensional $\mathbb{F}_{q^m}$-linear subspaces of $\mathbb{F}_{q^m}^n$ endowed with the \emph{rank distance}. We also write that $\mathcal{C}$  is an $\nkdqm$ code, where $d$ denotes its \textbf{minimum distance}, defined as
\[
d = d(\C) := \min \{ d(\mathbf{c}_1, \mathbf{c}_2) : \mathbf{c}_1, \mathbf{c}_2 \in \mathcal{C},\ \mathbf{c}_1 \ne \mathbf{c}_2 \}.
\]
If the minimum distance is not known or not relevant to the context, then we simply refer to $\mathcal{C}$ as an $[n,k]_{q^m/q}$ code. A \textbf{generator matrix} for $\mathcal{C}$ is a matrix $G \in \mathbb{F}_{q^m}^{k \times n}$ such that $\mathcal{C} = \{ \mathbf{x} G : \mathbf{x} \in \mathbb{F}_{q^m}^k \}$.
We say that an $[n,k]_{q^m/q}$ code is \textbf{nondegenerate} if the $\F_q$-linear column span of any generator matrix of $\C$ has dimension $n$. Finally, \textbf{the weight spectrum of a code $\C$} is defined as the set of all nonzero weights of its codewords, and it will be denoted by \[
\mathrm{WS}(\C):=\{\w(\cc) \colon \cc \in \C, \cc \neq \mathbf{0}\}.
\]

The focus of this paper is on studying the function $L_{\mathrm{rk}}(n,m,k,q)$, defined as the maximum number of distinct nonzero weights that a $k$-dimensional $\mathbb{F}_{q^m}$-linear rank-metric code in $\mathbb{F}_{q^m}^n$ can attain. We first note that, fixing only the parameters $k$, $q$, and $m$ is not sufficient to define a meaningful notion of maximum number of weights in the rank metric. This is due to the fact that the rank weight of any codeword is upper bounded by $m$, regardless of $n$. Moreover, readapting for linear rank-metric codes the same argument as in the proof of \cite[Proposition 2]{shi2019many}, we know that the number of distinct nonzero weights of an $[n,k]_{q^m/q}$ code is at most
\[
\frac{q^{km} - 1}{q^m - 1}.
\]
However, this upper bound exceeds $m$ when $k > 1$, and thus cannot be tight. In fact, the inequality
\[
m \geq \frac{q^{km} - 1}{q^m - 1} > q^{m(k-1)} \geq m(k-1)
\]
implies $k = 1$, which corresponds to trivial codes.

This shows that, unlike in the Hamming case, it is not possible to construct linear rank-metric codes with $\frac{q^{km} - 1}{q^m - 1}$ distinct weights when $k > 1$. Therefore, a direct analogue of maximum weight spectrum (MWS) codes does not naturally extend to the rank-metric setting. A more refined discussion of this topic will be given in Section~\ref{sec:conclusions}, where we discuss the definition of MWS codes in the rank-metric context. As a consequence, it becomes meaningful to fix the length $n$ and study the function $L_{\rk}(n, m, k, q)$. A first, trivial upper bound is given by
\[
L_{\rk}(n, m, k, q) \leq \min\{m, n\}.
\]

In analogy with the Hamming metric case, we define \textbf{full weight spectrum (FWS)} rank-metric codes $\C$ as those attaining this upper bound, or, in other words,\[\mathrm{WS}(\mathcal{C})=\{1,2,\ldots,\min\{m,n\}\}.\]

Hence, an FWS rank-metric code achieves the full range of nonzero weights available in the ambient space $\mathbb{F}_{q^m}^n$. 

Clearly, \( L_{\rk}(m, n, 1, q) = 1 \) and \( L_{\rk}(m, k, k, q) = \min\{k,m\} \). Therefore, in what follows, we assume \( n>k > 1 \).

We now present a family of linear rank-metric codes that will serve as a key tool for constructing examples with the maximum number of distinct nonzero weights.
Let $\lambda$ be a generator of the field extension $\mathbb{F}_{q^m}$ over $\mathbb{F}_q$, and let $\ell \leq m$ be a positive integer. We define the vector
\[
\mathbf{u}_{\lambda,\ell} := (1, \lambda, \lambda^2, \ldots, \lambda^{\ell-1}) \in \mathbb{F}_{q^m}^\ell.
\]
We denote by $U_{\lambda,\ell}$ the $\mathbb{F}_q$-subspace of $\mathbb{F}_{q^m}$ generated by the entries of $\mathbf{u}_{\lambda,\ell}$, i.e.
\[
U_{\lambda,\ell} := \langle 1, \lambda, \ldots, \lambda^{\ell-1} \rangle_{\mathbb{F}_q}.
\]
Now, let $\mathbf{n} = (n_1, \ldots, n_k)$ be a sequence of positive integers satisfying $m\geq  n_1 \geq \ldots \geq n_k \geq 1$. We define the block-diagonal matrix
\begin{equation} \label{eq:generatormatrixlambdageneral}
G_{\lambda, \mathbf{n}} := \begin{bmatrix}
    \uu_{\lambda,n_{1}}   &   \mathbf{0}       & \mathbf{0}          & \cdots    & \mathbf{0} \\
\mathbf{0}  &   \uu_{\lambda,n_{2}}     & \mathbf{0}          & \cdots    & \mathbf{0} \\
\mathbf{0}  & \mathbf{0}  & \uu_{\lambda,n_{3}}   & \cdots    & \mathbf{0} \\
\vdots  & \vdots    & \vdots    & \ddots    & \vdots\\
\mathbf{0}  & \mathbf{0}   & \cdots    & \mathbf{0}  & \uu_{\lambda,n_{k}}
\end{bmatrix}
\in \mathbb{F}_{q^m}^{k \times (n_1 + \ldots + n_k)}.
\end{equation}
We denote by $\mathcal{C}_{\lambda, \mathbf{n}}$ the rank-metric code having $G_{\lambda, \mathbf{n}}$ as a generator matrix. This is an $[n_1 + \ldots + n_k, k, n_k]_{q^m/q}$ code.

\begin{remark}
The codes $\mathcal{C}_{\lambda, \mathbf{n}}$ have recently appeared in different contexts and exhibit interesting properties.
\begin{itemize}
\item The matrix $G_{\lambda, \mathbf{n}}$ also serves as the parity-check matrix of a subclass of bounded-degree LRPC codes. Therefore, $\mathcal{C}_{\lambda, \mathbf{n}}$ can be interpreted as the dual of these codes. This particular class has been introduced in \cite{franch2024bounded} and shown in \cite{gaborit2013low} to outperform general LRPC codes in terms of decoding capability.
    
\item These codes $\mathcal{C}_{\lambda, \mathbf{n}}$ fall within the class of \emph{completely decomposable codes}, cf. \cite{santonastaso2024completely}. In particular, for certain choices of parameters, the codes $\mathcal{C}_{\lambda, \mathbf{n}}$ attain the maximum possible number of minimum weight codewords among all completely decomposable codes. Also, the codes $\mathcal{C}_{\lambda, \mathbf{n}}$ are proven to have the largest number of codewords with maximum rank weight among a certain class of rank-metric codes \cite{polverino2023maximum}. Additional metric properties of $\mathcal{C}_{\lambda, \mathbf{n}}$ are discussed in \cite[Section~4.1]{santonastaso2024completely}.

\end{itemize}
\end{remark}

We now introduce an auxiliary function $\Psi$, which will help us to choose the vector $\mathbf{n}$ in \eqref{eq:generatormatrixlambdageneral} in a way which guarantees that the code $\C_{\lambda,\mathbf{n}}$ has a weight spectrum with a large size.

\begin{definition} \label{def:definitionPsi}
    Let \( u, v \) be positive integers with \( u \geq v \geq 2\). Define the function
    \[
        \Psi(u, v) := (u_1, u_2, \ldots, u_v),
    \]
    where the entries \( u_i \) are defined recursively as follows:
    \begin{align*}
        u_1 &:= \left\lceil \frac{u}{2} \right\rceil, \\
        u_i &:= \left\lceil \frac{\left\lfloor \frac{u}{2^{i-1}} \right\rfloor}{2} \right\rceil \quad \text{for } i = 2, \ldots, v - 1, \\
        u_v &:= \left\lfloor \frac{u}{2^{v-1}} \right\rfloor.
    \end{align*}
\end{definition}

We start to investigate the properties of the function $\Psi$ defined above.

\begin{proposition} \label{prop:propertiesofPSI}
    Let \( u, v \) be positive integers with \( u \geq v \), and let $\Psi(u, v) = (u_1, u_2, \ldots, u_v)$. Then the following properties hold:
    \begin{enumerate}[label=(\roman*)]
        \item $
            \sum\limits_{i=1}^v u_i = u;
        $
        \item $
            u_i \in \left\{ \sum\limits_{j=i+1}^{v} u_j,\ \sum\limits_{j=i+1}^{v} u_j + 1 \right\}, \quad \text{for all } i \in \{1, \ldots, v-1\}.
        $
    \end{enumerate}
\end{proposition}

\begin{proof}
Recall that for any positive integer \( h \), one has the identity
\[
h = \left\lceil \frac{h}{2} \right\rceil + \left\lfloor \frac{h}{2} \right\rfloor,
\]
and also that
\[
\left\lfloor \frac{\left\lfloor \frac{h}{i} \right\rfloor}{j} \right\rfloor = \left\lfloor \frac{h}{ij} \right\rfloor
\]
for any positive integers \( i, j \).\\
We now prove (i). By definition of the function \( \Psi(u, v) = (u_1, \ldots, u_v) \), we have:
\[
\begin{aligned}
\sum_{i=1}^v u_i 
&= \sum_{i=1}^{v-2} u_i + u_{v-1} + u_v \\
&= \sum_{i=1}^{v-2} u_i + \left\lceil \frac{\left\lfloor \frac{u}{2^{v-2}} \right\rfloor}{2} \right\rceil + \left\lfloor \frac{u}{2^{v-1}} \right\rfloor \\
&= \sum_{i=1}^{v-2} u_i + \left\lceil \frac{\left\lfloor \frac{u}{2^{v-2}} \right\rfloor}{2} \right\rceil + \left\lfloor \frac{\left\lfloor \frac{u}{2^{v-2}} \right\rfloor}{2} \right\rfloor \\
&= \sum_{i=1}^{v-2} u_i + \left\lfloor \frac{u}{2^{v-2}} \right\rfloor \\
&\quad \vdots \\
&= u_1 + \left\lfloor \frac{u}{2} \right\rfloor \\
&= \left\lceil \frac{u}{2} \right\rceil + \left\lfloor \frac{u}{2} \right\rfloor = u.
\end{aligned}\]
This proves (i). \\ We now prove (ii). Let \( i \in \{1, \ldots, v-1\} \). First, we claim that
\begin{equation} \label{eq:ceilfloor}
\left\lceil \frac{\left\lfloor \frac{u}{2^{i-1}} \right\rfloor}{2} \right\rceil \in \left\{ \left\lfloor \frac{u}{2^{i}} \right\rfloor,\ \left\lfloor \frac{u}{2^{i}} \right\rfloor + 1 \right\}.
\end{equation}
Indeed, we have:
\[\begin{aligned} 
\left\lfloor \frac{u}{2^{i}} \right\rfloor 
= \left\lfloor \frac{\left\lfloor \frac{u}{2^{i-1}} \right\rfloor}{2} \right\rfloor \nonumber\leq \left\lceil \frac{\left\lfloor \frac{u}{2^{i-1}} \right\rfloor}{2} \right\rceil \nonumber\leq \frac{\left\lfloor \frac{u}{2^{i-1}} \right\rfloor + 1}{2} \leq  \left\lfloor \frac{u}{2^{i}} \right\rfloor + 1. \nonumber
\end{aligned}\]
Since the value $\left\lceil \frac{\left\lfloor \frac{u}{2^{i-1}} \right\rfloor}{2} \right\rceil$  lies between \( \left\lfloor \tfrac{u}{2^{i}} \right\rfloor \) and \( \left\lfloor \tfrac{u}{2^{i}} \right\rfloor + 1 \), and it is an integer, it must be one of these two values. Moreover, arguing as in (i), we have that for $i=1,\dots, v-2$
\[
\begin{aligned}
\sum_{j=i+1}^v u_j =& \sum_{j=i+1}^{v-2}u_j +u_{v-1}+u_v\\
=&\sum_{j=i+1}^{v-2}u_j +\left\lfloor\frac{u}{2^{v-2}}\right\rfloor\\
\vdots\\
=&u_{i+1}+\left\lfloor\frac{u}{2^{i+1}}\right\rfloor\\
=&\left\lceil\frac{\lfloor\frac{u}{2^i}\rfloor}{2}\right\rceil+\left\lfloor\frac{\lfloor\frac{u}{2^i}\rfloor}{2}\right\rfloor\\
=&\left\lfloor \frac{u}{2^i} \right\rfloor.    
\end{aligned}
\]
Therefore, since 
\[
u_i = \left\lceil \frac{\left\lfloor \tfrac{u}{2^{i-1}} \right\rfloor}{2} \right\rceil \text{ for }i=1,\dots, v-1
\]
by~\eqref{eq:ceilfloor}, the claim follows.

\end{proof}

Our general strategy for proving the Main Theorem will be to distinguish between the case $n \leq m$, which we examine in Section \ref{sec:nleqm}, and the case $n>m$, which we examine in Section \ref{sec:n>m}. Such a distinction can be understood intuitively by noticing that a codeword $\cc \in \Fqmn$ in the case $n \leq m$ can have any weight in $\{1, \dots, n\}$, whereas in the case $n > m$ the available weights for a single codeword are $\{1, \dots, m\}$, because the dimension of the field extension limits the available weight spectrum. In both cases, with the aid of the function $\Psi$, we will present an explicit construction of codes attaining the equality in the Main Theorem. Moreover, we will prove that the number of non-zero weights reached by these constructions is the best possible. To do so, we will use combinatorial lemmas and, in the case $n>m$, we will involve also the geometric interpretation of $\Fqm$-linear rank metric codes.

\section{The case when \texorpdfstring{$n \leq m$}{n <= m}}
\label{sec:nleqm}

We start our investigation by studying the case $n \leq m$. In this section, we establish the following result, which corresponds to the case $n \leq m$ in the Main Theorem.

\begin{theorem}\label{thm:small_n}
Assume that $n \leq m$. Let $s := \max\left\{\floor*{\frac{n}{2^{k-1}}}, 1\right\}$. Then
$$L_{\rk}(n, m, k, q) = n - s+ 1.$$
\end{theorem}

We will establish Theorem \ref{thm:small_n} by showing that the quantity $n - s + 1$ serves as both a lower and an upper bound for $L_{\rk}(n, m, k, q)$. As first, we provide an explicit construction of a code with exactly $n-s+1$ nonzero weights. This will also imply that the value $n-s+1$ is a lower bound for $L_{\rk}(n,m,k,q)$.\\

We start with a simple case, i.e. $k=2$, in order to better illustrate the ideas behind our construction.

\begin{example} \label{ex:k2}
Assume $k = 2$. We consider the code $\C_{\lambda,\mathbf{n}}$ generated by the matrix $G_{\lambda,\mathbf{n}}$ as defined in~\eqref{eq:generatormatrixlambdageneral}, namely:
$$G_{\lambda,\mathbf{n}} = \begin{bmatrix}
\uu_{\lambda, a} & \mathbf{0}  \\
\mathbf{0}  & \uu_{\lambda, b}
\end{bmatrix},$$
where $\mathbf{n} = (a, b)$ with $a \geq b$, $m \geq n = a + b$, and $\lambda$ is a generator of $\F_{q^m}$ over $\F_q$. Recall that $\uu_{\lambda, \ell} = (1, \lambda, \ldots, \lambda^{\ell - 1})$, as defined previously. The code $\C_{\lambda,\mathbf{n}}$ is therefore a $[a + b, 2, b]_{q^m/q}$ code generated by the codewords $\mathbf{c}_1 = (\uu_{\lambda, a}, \mathbf{0})$ and $\mathbf{c}_2 = (\mathbf{0}, \uu_{\lambda, b})$ in $\F_{q^m}^n$. Note that $\w(\mathbf{c}_1) = a$ and $\w(\mathbf{c}_2) = b$, so the weight spectrum of $\C_{\lambda,\mathbf{n}}$ includes the values $a$ and $b$.
Every codeword of $\C_{\lambda,\mathbf{n}}$ is of the form \[\alpha_1\mathbf{c}_1+\alpha_2\mathbf{c}_2=(\alpha_1 \uu_{\lambda,a}, \alpha_2 \uu_{\lambda,b} ),\]
for $\alpha_1, \alpha_2 \in \F_{q^m}$. Note that because of the block diagonal form of $G_{\lambda, \bf{n}}$, any codeword of the form $\alpha_1\mathbf{c}_1+\alpha_2\mathbf{c}_2$ with $\alpha_{1} \neq 0$ must have rank weight at least $a$. In particular, for $1 \leq i \leq b$, the codeword $\lambda^i\mathbf{c}_1+\mathbf{c}_2=(\lambda^i \uu_{\lambda, a}, \uu_{\lambda, b})$ has weight $a + i$. Indeed, we have:
\[
\begin{array}{rl}
   \w((\lambda^i,1)G_{\lambda,\bf{n}}) & = \w((\lambda^i \uu_{\lambda, a}, \uu_{\lambda, b})) \\
    & = \dim_{\F_q}\left(\lambda^i U_{\lambda,a}+U_{\lambda,b} \right) \\
     & = \dim_{\F_q}(\langle 1,\ldots,\lambda^{b-1},\lambda^i,\ldots,\lambda^{a+i-1} \rangle_{\F_q}) \\
     & =a+i.
\end{array}
\]
Therefore, the integers $a + 1, \ldots, a + b$ also belong to the weight spectrum. Therefore, the weight spectrum of $\C_{\lambda,\mathbf{n}}$ is exactly the set \[
\mathrm{WS}(\C)=\{b, a, a+1, \ldots, n\}.\]

Now assume that $a$ and $b$ can vary, while keeping $a+b = n$.
The size of the weight spectrum is maximal in the case $a = \lceil n/2 \rceil$ and $b = \lfloor n/2 \rfloor$, i.e. $(a,b)=\Psi(n,2)$.

$\hfill \lozenge$
\end{example}

We now consider a special instance of the construction described in Example~\ref{ex:k2}. This instance yields codes with exactly \( n - s + 1 \) distinct nonzero weights.

\begin{example}
Assume $n = 7$ and $k = 2$, and let $m \geq 7$. 
Consider the $[7, 2, 3]_{q^m/q}$ code $\C_{\lambda, \mathbf{n}}$ with $\mathbf{n} = (4,3)$, generated by the matrix
\[
G=G_{\lambda,\mathbf{n}} = \begin{bmatrix} 
\uu_{\lambda,4} & \mathbf{0} \\
\mathbf{0}  & \uu_{\lambda,3}
\end{bmatrix} = \begin{bmatrix} 
1 & \lambda & \lambda^2 & \lambda^3 & 0 & 0 & 0\\
0 & 0 & 0 & 0 & 1 & \lambda & \lambda^2
\end{bmatrix},
\]
where $\lambda \in \F_{q^m}$ is a generator of $\F_{q^m}$ over $\F_q$. This corresponds to the costruction as in Example~\ref{ex:k2} with $a = 4$ and $b = 3$, i.e. $(a,b)=\Psi(n,2)$. The code $\C_{\lambda,\mathbf{n}}$ has exactly $5$ distinct nonzero weights. Indeed,
\[
\begin{aligned}
&\w((0,1)G)=\w((0,0,0,0,1,\lambda,\lambda^2))=\dim_{\F_q}(\langle 1, \lambda,\lambda^2\rangle_{\F_q})=3\\
&\w((1,0)G)=\w((1,\lambda,\lambda^2,\lambda^3,0,0,0))=\dim_{\F_q}(\langle 1, \lambda,\lambda^2,\lambda^3\rangle_{\F_q})=4\\
&\w((\lambda,1)G)=\w((\lambda,\lambda^2,\lambda^3,\lambda^4,1,\lambda,\lambda^2))=\dim_{\F_q}(\langle 1, \lambda,\lambda^2,\lambda^3,\lambda^4\rangle_{\F_q})=5\\
&\w((\lambda^2,1)G)=\w((\lambda^2,\lambda^3,\lambda^4,\lambda^5,1,\lambda,\lambda^2))=\dim_{\F_q}(\langle 1,\lambda,\lambda^2,\lambda^3,\lambda^4,\lambda^5\rangle_{\F_q})=6\\
&\w((\lambda^3,1)G)=\w((\lambda^3,\lambda^4,\lambda^5,\lambda^6,1,\lambda,\lambda^2))=\dim_{\F_q}(\langle 1, \lambda,\lambda^2,\lambda^3,\lambda^4,\lambda^5,\lambda^6\rangle_{\F_q})=7.
\end{aligned}
\]

Thus, the weight spectrum of $\C_{\lambda,\mathbf{n}}$ is precisely $\{3, 4, 5, 6, 7\}$, which has size $5 = n - s + 1$, where the value of $s$ is as in Theorem \ref{thm:small_n}, i.e. 
\[
s = \left\lfloor \frac{n}{2^{k-1}} \right\rfloor = \left\lfloor \frac{7}{2} \right\rfloor = 3.
\]

$\hfill \lozenge$
\end{example}

We now describe our general construction. This is based on considering a code with a generator matrix as in \eqref{eq:generatormatrixlambdageneral}, where the vector $\mathbf{n}$ is determined by using $\Psi(u,v)$ defined in Definition \ref{def:definitionPsi}. This will guarantee that the code described in this way has $n-s+1$ non-zero weights.

\begin{construction}
\label{constr:casenleqm}
Assume \( k < n \leq m \). Define the set
\[
S := \left\{ i \in \Z_{\geq 0} \colon \frac{n - i}{2^{k - i - 1}} \geq 1 \right\}.
\]
Note that \( k - 1 \in S \), and let \( z := \min S \leq k-1\). \\
\medskip
\noindent
\textbf{\underline{Case \( z = 0 \).}} Consider the tuple \( \mathbf{n} := \Psi(n, k) = (n_1, \ldots, n_k) \), where the function $\Psi$ is as in Definition \ref{def:definitionPsi}, and define the $[n,k]_{q^m/q}$ code \( \mathcal{C}_{\lambda, \mathbf{n}} \) with generator matrix \( G_{\lambda, \mathbf{n}} \) as in~\eqref{eq:generatormatrixlambdageneral}, where \( \lambda \in \F_{q^m} \) is a generator of \( \F_{q^m} \) over \( \F_q \), and each block \( \uu_{\lambda, \ell} \) is defined as \( (1, \lambda, \lambda^2, \ldots, \lambda^{\ell - 1}) \).

\medskip
\noindent
\textbf{\underline{Case \( z > 0 \).}} Let \( n' := n - z \) and \( k' := k - z \). Consider the tuple \( \Psi(n', k') = (n_1, \ldots, n_{k'}) \), and define the extended tuple
\[
\mathbf{n} := (n_1, \ldots, n_{k'}, \underbrace{1, \ldots, 1}_{z\ \text{times}}).
\]
Note that $n_{k'} = \lfloor n'/2^{k'-1} \rfloor=\lfloor n-z/2^{k-z-1} \rfloor \geq 1$. Then define the $[n,k]_{q^m/q}$ code \( \mathcal{C}_{\lambda, \mathbf{n}} \) with generator matrix \( G_{\lambda, \mathbf{n}} \) as in~\eqref{eq:generatormatrixlambdageneral}.
\end{construction}

We now prove that the code $\C_{\lambda, \bf{n}}$ defined above has exactly $n-s+1$ distinct weights.

\begin{theorem}
\label{thm:constrnleqm}
Let \( n \leq m \), and let \( s = \max\left\{ \left\lfloor \frac{n}{2^{k-1}} \right\rfloor, 1 \right\} \). The $[n,k]_{q^m/q}$ code $\mathcal{C}_{\lambda,\mathbf{n}}$ defined as in Construction \ref{constr:casenleqm} has weight spectrum
\[
\mathrm{WS}(\C_{\lambda,\mathbf{n}})=\{s,s+1,\ldots,n\}
\]
and therefore has precisely \( n - s + 1 \) distinct nonzero weights.
\end{theorem}

\begin{proof}
We divide the discussion in two parts according $z$, defined as in Construction \ref{constr:casenleqm}, is 0 or greater than $0$.

\noindent\underline{Case $z=0$}. We prove that the code $\mathcal{C}_{\lambda,\mathbf{n}}$ has weight spectrum \( \{n_k, \ldots, n\} \). Since \( n_k = s = \left\lfloor \frac{n}{2^{k-1}} \right\rfloor \), this will imply the assertion.

\medskip

For every \( i = 1, \dots, k \), let \( \mathbf{c}_{n_i} = (\mathbf{0}, \ldots, \uu_{\lambda,n_i}, \ldots, \mathbf{0}) \in \mathcal{C}_{\lambda,\mathbf{n}} \), where $\mathbf{c}_{n_i}$ is written as a vector defined by blocks of length $n_1,n_2,\dots,n_k$ and \( \uu_{\lambda,n_i} \) appears in the \( i \)-th block. Note that \( n_k \leq n_{k-1} \leq \dots \leq n_1 \), and the codeword \( \mathbf{c}_{n_k} = (\mathbf{0}, \ldots, \mathbf{0}, \uu_{\lambda,n_k}) \in \mathcal{C}_{\lambda,\mathbf{n}} \) has rank weight
\[
\w(\mathbf{c}_{n_k}) = \dim_{\F_q}(\langle \uu_{\lambda,n_k} \rangle_{\F_q}) = \dim_{\F_q}(\langle 1, \lambda, \ldots, \lambda^{n_k - 1} \rangle_{\F_q}) = n_k,
\]
which is the minimum possible.

\medskip

Now, we show how to generate codewords covering all weights from \( n_k \) up to \( n \). For every \( j = 0, \ldots, n_k \), consider:
\[
\mathbf{c}_{n_k} + \lambda^j \mathbf{c}_{n_{k-1}} = (\mathbf{0}, \ldots, \mathbf{0}, \lambda^j \uu_{\lambda,n_{k-1}}, \uu_{\lambda,n_k}) \in \mathcal{C}_{\lambda,\mathbf{n}}.
\]
Then, since $n_k\leq n_{k-1}$, we have that 
\[
\w(\mathbf{c}_{n_k} + \lambda^j \mathbf{c}_{n_{k-1}}) = \dim_{\F_q}(\langle 1, \lambda, \ldots,\lambda^{n_k-1},\ldots, \lambda^{n_{k-1} + j - 1} \rangle_{\F_q}) = n_{k-1} + j.
\]
By Proposition~\ref{prop:propertiesofPSI}(ii), we know that either
\[
n_{k-1} = n_k \quad \text{or} \quad n_{k-1} = n_k + 1.
\]
So, by letting \( j \) vary from \( 0 \) to \( n_k \), we obtain codewords of weights \(  n_k + 1, \ldots, n_{k-1} + n_k \). Again by Proposition~\ref{prop:propertiesofPSI}(ii), we have
\[
\text{ either } n_{k-1} + n_k=n_{k-2}  \quad \text{or} \quad n_{k-1} + n_k= n_{k-2}-1.
\]
Therefore, we cover all weights up to either $n_{k-2}-1$ or $n_{k-2}$.

\noindent Now, for each \( j = 0, \ldots, n_{k-1} + n_k \), consider:
\[
\mathbf{c}_{n_k} + \lambda^{n_k} \mathbf{c}_{n_{k-1}} + \lambda^j \mathbf{c}_{n_{k-2}} = (\mathbf{0}, \ldots, \lambda^j \uu_{\lambda,n_{k-2}}, \lambda^{n_k} \uu_{\lambda,n_{k-1}}, \uu_{\lambda,n_k}) \in \mathcal{C}_{\lambda,\mathbf{n}}.
\]
Then,
\[
\begin{aligned}
\w(\mathbf{c}_{n_k} + \lambda^{n_k} \mathbf{c}_{n_{k-1}} + \lambda^j \mathbf{c}_{n_{k-2}}) &= \dim_{\F_q}(\langle 1, \lambda, \ldots,\lambda^{n_k-1},\lambda^{n_k}, \dots, \lambda^{n_k+n_{k-1}-1},\dots, \lambda^j, \dots, \lambda^{n_{k-2} + j - 1} \rangle_{\F_q})\\ &= n_{k-2} + j.
\end{aligned}
\]
So, we find codewords of weights \( n_{k-2}, \ldots, n_{k-2} + n_{k-1} + n_k \). Again by Proposition~\ref{prop:propertiesofPSI}(ii), we know that this covers weights up to \( n_{k-3} \) or \( n_{k-3} - 1 \) since \[
\text{ either } n_{k-2} + n_{k-1} + n_k=n_{k-3}-1 \text{ or } n_{k-2} + n_{k-1} + n_k=n_{k-3}.\]

Proceeding recursively in this way, after a finite number of steps, we obtain codewords of every weight from \( n_k \) to \( n_2+n_3+\dots+n_k \), where:
\[
\text{ either } n_1+n_2+\dots+n_k=n_1-1  
\text{ or } n_1+n_2+\dots+n_k=n_1.
\]

Finally, for every \( j = 0, \ldots, n_2 + \cdots + n_k \), consider:
\[
\mathbf{c}_{n_k} + \lambda^{n_k} \mathbf{c}_{n_{k-1}} + \lambda^{n_{k-1}} \mathbf{c}_{n_{k-2}} + \cdots + \lambda^{n_3} \mathbf{c}_{n_2} + \lambda^j \mathbf{c}_{n_1} \in \mathcal{C}_{\lambda, \mathbf{n}}.
\]
Then the weight is:
\[
\begin{aligned}
&\w(\mathbf{c}_{n_k} + \lambda^{n_k} \mathbf{c}_{n_{k-1}} + \lambda^{n_{k-1}} \mathbf{c}_{n_{k-2}} + \cdots + \lambda^{n_3} \mathbf{c}_{n_2} + \lambda^j \mathbf{c}_{n_1})=\\
&=\dim_{\F_q}(\langle 1, \lambda, \ldots,\lambda^{n_k-1},\lambda^{n_k},\dots, \lambda^{n_{k-1}-1}, \lambda^{n_{k-1}}, \dots,\lambda^j,\dots, \lambda^{n_1 + j - 1} \rangle_{\F_q}) \\
&=n_1 + j.
\end{aligned}
\]

Thus, letting \( j \) vary, we obtain codewords of every weight from \( n_1 \) up to \( n = n_1 + \cdots + n_k \), by (i) of Proposition~\ref{prop:propertiesofPSI}.

Hence, the weight spectrum of \( \mathcal{C}_{\lambda,\mathbf{n}} \) is exactly \( \{n_k, n_k + 1, \ldots, n\} \), and its size is \( n - n_k + 1 \), as claimed.

\medskip
\noindent \underline{Case $z>0$}. First, we note that
\[
n_{k'} = \left\lfloor \frac{n'}{2^{k' - 1}} \right\rfloor = \left\lfloor \frac{n - z}{2^{k - z - 1}} \right\rfloor = 1.
\]
Indeed, if
\[
\left\lfloor \frac{n - z}{2^{k - z - 1}} \right\rfloor \geq 2,
\]
then
\[
\frac{n - z}{2^{k - z - 1}} \geq 2,
\]
which implies
\[
\frac{n - z + 1}{2^{k - z}} \geq 1,
\]
and so
\[
\left\lfloor \frac{n - z + 1}{2^{k - z}} \right\rfloor \geq 1.
\]
This implies that \( z' := z - 1 \in S \), where \( S \) is as defined in Construction~\ref{constr:casenleqm}, contradicting the minimality of \( z \).

Now consider the punctured code of \( \mathcal{C} \), defined as
\[
\mathcal{C}' = \left\{ (\alpha_{n_1} \uu_{n_1}, \ldots, \alpha_{n_{k'}} \uu_{n_{k'}}) \mid \alpha_{n_i} \in \F_{q^m}, i=1,\dots,k' \right\}.
\]
The code \( \mathcal{C}' \) is the \( [n - z, k - z]_{q^m/q} \) code, constructed as in Construction~\ref{constr:casenleqm}, and for which
\[
z' = \min \left\{ i \in \Z_{\geq 0} \colon \frac{n - z - i}{2^{k - z - i - 1}} \geq 1 \right\} = 0.
\]
Applying the previous case to \( \mathcal{C}' \), we obtain
\[
\mathrm{WS}(\mathcal{C}') = \{n_{k'}, \ldots, n'\} = \{1, \ldots, n - z\}.
\]
Hence, \( |\mathrm{WS}(\mathcal{C}')| = n' - n_{k'} + 1 = n' = n - z \). Furthermore, every codeword \( \overline{\mathbf{c}} \in \mathcal{C}' \) corresponds to the codeword \( (\overline{\mathbf{c}}, \mathbf{0}) \in \mathcal{C} \), so we conclude that
\[
\{1, \ldots, n - z\} = \mathrm{WS}(\mathcal{C}') \subseteq \mathrm{WS}(\mathcal{C}).
\]

\medskip

It remains to show that there are codewords in \( \mathcal{C} \) of weights \( n - z + 1, \ldots, n \). By construction, there exists \( \mathbf{c}' \in \mathcal{C}' \) such that
\[
\langle \mathbf{c}' \rangle_{\F_q} = \langle 1, \lambda, \ldots, \lambda^{n - z - 1} \rangle_{\F_q}, \quad \text{and} \quad \w(\mathbf{c}') = n - z.
\]
Now, consider the family of codewords of the form
\[
(\mathbf{c}', \lambda^{i_1}, \ldots, \lambda^{i_z}) \in \mathcal{C},
\]
for arbitrary non-negative integers \( i_1, \ldots, i_z \).

By choosing tuples of the form:
\begin{align*}
(i_1, \ldots, i_z) &= (n - z, 0, \ldots, 0), \\
(i_1, \ldots, i_z) &= (n - z, n - z + 1, 0, \ldots, 0), \\
&\vdots \\
(i_1, \ldots, i_z) &= (n - z, n - z + 1, \ldots, n - 1, 0), \\
(i_1, \ldots, i_z) &= (n - z, n - z + 1, \ldots, n - 1, n),
\end{align*}
we generate codewords with weights ranging from \( n - z +1\) up to \( n \). Therefore, the weight spectrum of \( \mathcal{C} \) is exactly \( \{1, \ldots, n\} \), and the assertion is proved.

\end{proof}

In the next, we show examples of codes defined as in Construction \ref{constr:casenleqm} and we show that their weight spectrums have size $n-s+1$.

\begin{example}
Let \( n = 8 \), \( m = 8 \), and \( k = 3 \). Since \( n = 8 \geq 2^k = 8 \), we have that $s=\lfloor\frac{n}{2^{k-1}}\rfloor=2>1$ and, by \Cref{thm:constrnleqm}, Construction~\ref{constr:casenleqm} provides a code with $n-s+1=8-2+1=7$ non-zero weights. In this case, we have \( z = 0 \), and Construction~\ref{constr:casenleqm} gives the code \( \C_{\lambda,\mathbf{n}} \) with \( \mathbf{n} = (n_1, n_2, n_3) = (4, 2, 2) \), and generator matrix
\[
G_{\lambda, \mathbf{n}} = 
\begin{bmatrix}
\uu_{\lambda, 4} & {\bf 0} & {\bf 0} \\
{\bf 0} & \uu_{\lambda, 2} & {\bf 0} \\
{\bf 0} & {\bf 0} & \uu_{\lambda, 2}
\end{bmatrix}.
\]
This code has codewords of every weight from 2 to 8. Indeed, following the proof of Theorem \ref{thm:constrnleqm}, we have that 
\begin{itemize}
\item \( \mathbf{c}_{n_3}=(0,0,0,0,0,0, \uu_{\lambda, 2})=(0,0,0,0,0,0, 1,\lambda) \) has weight 2.
    \item $\mathbf{c}_{n_3} + \lambda \mathbf{c}_{n_2} = (0,0,0,0, \lambda\uu_{\lambda,2}, \uu_{\lambda,2})=(0,0,0,0,\lambda,\lambda^2,1,\lambda)$ has weight 3.
    \item $\mathbf{c}_{n_3} + \lambda^2\mathbf{c}_{n_2} = (0,0,0,0, \lambda^2\uu_{\lambda,2}, \uu_{\lambda,2})=(0,0,0,0,\lambda^2,\lambda^3,1,\lambda)$ has weight 4.
    \item $\mathbf{c}_{n_3} + \lambda^2 \mathbf{c}_{n_2}+\lambda\mathbf{c}_{n_1}=(\lambda\uu_{\lambda,4}, \lambda^2\uu_{\lambda,2}, \uu_{\lambda, 2})=(\lambda,\lambda^2, \lambda^3,\lambda^4,\lambda^2,\lambda^3,1,\lambda)$ has weight 5.
    \item $\mathbf{c}_{n_3} + \lambda^2 \mathbf{c}_{n_2}+\lambda^2\mathbf{c}_{n_1}=(\lambda^2\uu_{\lambda,4}, \lambda^2\uu_{\lambda,2}, \uu_{\lambda, 2})=(\lambda^2,\lambda^3,\lambda^4,\lambda^5,\lambda^2,\lambda^3,1,\lambda)$ has weight 6.
    \item $\mathbf{c}_{n_3} + \lambda^2 \mathbf{c}_{n_2}+\lambda^3\mathbf{c}_{n_1}=(\lambda^3\uu_{\lambda,4}, \lambda^2\uu_{\lambda,2}, \uu_{\lambda, 2})=(\lambda^3,\lambda^4,\lambda^5,\lambda^6,\lambda^2,\lambda^3,1,\lambda)$ has weight 7.
    \item $\mathbf{c}_{n_3} + \lambda^2 \mathbf{c}_{n_2}+\lambda^4\mathbf{c}_{n_1}=(\lambda^4\uu_{\lambda,4}, \lambda^2\uu_{\lambda,2}, \uu_{\lambda, 2})=(\lambda^4,\lambda^5,\lambda^6,\lambda^7,\lambda^2,\lambda^3,1,\lambda)$ has weight 8.
\end{itemize}
\medskip

\noindent We now consider the case \( n = 6 \), \( m = 6 \), and \( k = 4 \). Since \( n = 6 \leq 2^{k-1} = 8 \), we have that $z=1>0$ and $s=\max\lbrace \lfloor\frac{n}{2^{k-1}}\rfloor,1\rbrace=1$ and
Construction~\ref{constr:casenleqm} provides a code with $n-s+1=6$ non-zero weights. Precisely, Construction~\ref{constr:casenleqm} gives the code \( \C_{\lambda,\mathbf{n}} \) with \( \mathbf{n} = (n_1, n_2, n_3,n_4) = (3, 1, 1, 1) \), and generator matrix
\[
G_{\lambda, \mathbf{n}} = 
\begin{bmatrix}
\uu_{\lambda, 3} & {\bf 0} & {\bf 0} & {\bf 0} \\
{\bf 0} & \uu_{\lambda, 1} & {\bf 0} & {\bf 0} \\
{\bf 0} & {\bf 0} & \uu_{\lambda, 1} & {\bf 0} \\
{\bf 0} & {\bf 0} & {\bf 0} & \uu_{\lambda, 1}
\end{bmatrix}.
\]

As in the previous example, one can check that \( \C_{\lambda,\mathbf{n}} \) contains codewords of every rank weight from \( 1 \) to \( 6 \). Indeed, following the proof of Theorem \ref{constr:casenleqm}, consider the punctured code of $\mathcal{C}$, defined as \[
\mathcal{C}' = \left\{ (\alpha_{n_1} \uu_{n_1}, \alpha_{n_2} \uu_{n_2}, \alpha_{n_{3}} \uu_{n_{3}}) \mid \alpha_i \in \F_{q^m}, i=1,2,3 \right\}
\] Then
\begin{itemize}
    \item ${\bf c}'_1={\bf c}'_{n_3}=(0,0,0,0, \uu_{\lambda, 1})=(0,0,0,0,1)\in\mathcal{C}'$ has weight 1.
    \item ${\bf c}'_2={\bf c}'_{n_3}+\lambda{\bf c}'_{n_2}=(0,0,0,\lambda\uu_{\lambda, 1}, \uu_{\lambda, 1})=(0,0,0,\lambda,1)\in\mathcal{C}'$ has weight 2.
    \item ${\bf c}'_3={\bf c}'_{n_3}+\lambda {\bf c}'_{n_2}+ {\bf c}'_{n_1}=(\uu_{\lambda,3}, \lambda^2\uu_{\lambda, 1}, \uu_{\lambda, 1})=(1, \lambda,\lambda^2, \lambda, 1)\in\mathcal{C}'$ has weight 3.
    \item ${\bf c}'_4={\bf c}'_{n_3}+\lambda {\bf c}'_{n_2}+ \lambda{\bf c}'_{n_1}=(\lambda\uu_{\lambda,3}, \lambda\uu_{\lambda, 1}, \uu_{\lambda, 1})=( \lambda,\lambda^2,\lambda^3, \lambda, 1)\in\mathcal{C}'$ has weight 4.
    \item ${\bf c}'_5={\bf c}'_{n_3}+\lambda {\bf c}'_{n_2}+\lambda^2 {\bf c}'_{n_1}=(\lambda^2\uu_{\lambda,3}, \lambda^2\uu_{\lambda, 1}, \uu_{\lambda, 1})=(\lambda^2, \lambda^3,\lambda^4, \lambda, 1)\in\mathcal{C}'$ has weight 5.
\end{itemize}
Then the codewords ${\bf c}_j=({\bf c}'_j, 0)\in\mathcal{C}$ for every $j=1,\dots, 5$ and so $\mathcal{C}$ has codewords of weight from 1 up to 5. Finally, consider ${\bf c}_6=({\bf{c}}_5,\lambda^5)=(\lambda^2,\lambda^3,\lambda^4,\lambda,1,\lambda^5)\in\mathcal{C}$ has weight 6. Thus $\mathcal{C}$ in an $[6,4]_{q^6/q}$ code with $6$ non-zero weights.

$\hfill \lozenge$
\end{example}

As immediate consequence of Theorem \ref{thm:constrnleqm}, we get the following corollary. 

\begin{corollary}
\label{cor:lowerboundnleqm}
Assume that $n \leq m$. Let $s = \max\left\{\floor*{\frac{n}{2^{k-1}}}, 1\right\}$. Then
\[ L_{\rk}(n, m, k, q) \geq n - s+ 1.\]    
\end{corollary}

Now, to conclude the proof of Theorem \ref{thm:small_n}, we will show that $n-s+1$ is also an upper bound for $L_{\rk}(n,m,k,q)$. \\

Note that when $s=1$, which means $n<2^k$, we already know that $L_{\rk}(n,m,k,q) = n$. To examine the case in which $s>1$, which corresponds to $n\geq 2^{k}$, we will use the following combinatorial lemma.

\begin{lemma}
\label{lem:combinatoriallemma}
Assume $n \geq 2^k$ and let $s = \floor*{\frac{n}{2^{k-1}}}>1$. Let $S$ be a subset of $\{1, \dots, n\}$ such that $|S| \geq n-s+2$. Then there exist $k+1$ elements $s_{0}, \dots, s_{k} \in S$ such that
$$\quad s_{j} > \sum_{i=0}^{j-1} s_{i}, \ \ \ \mbox{ for every } 1 \leq j \leq k.$$
\end{lemma}

\begin{proof}
Define $s_{0} = \min S$. Clearly, there are at most $s-2$ elements of $\{1, \dots, n\}$ that are not in $S$. Since $\{1, \dots, s_{0}-1 \} \cap S=\emptyset$, there are already $s_{0}-1$ elements which are not in $S$. As a consequence, there are at most $(s-2) - (s_{0} - 1) = s-s_{0}-1$ elements in $\{s_{0}+1, \dots, n\}$ which do not belong to $S$.\\
Now consider the interval \( I_1 = \{s_0 + 1, \dots, s\} \). Since $|I_{1}| = s-s_{0}$, there exists an element in $S \cap I_{1}$. Let us call this element $s_{1}$. Clearly, $s_{1} > s_{0}$. \\
Next, consider \( I_2 = \{s + s_0 + 1, \dots, 2s\} \). By similar reasoning, there must be some \( s_2 \in S \cap I_2 \), and we have \( s_2 > s_0 + s_1 \).\\
It is possible to do this for every interval of the form
\[
I_j = \left\{ s_0 + (2^{j-1} - 1)s + 1, \dots, 2^{j-1}s \right\},
\]
for \( j = 1, \dots, k \). Note that $j = k$ is the last value for which $2^{j-1}s \leq n$ still holds, since $2^{k-1}s=2^{k-1}\lf\frac{n}{2^{k-1}}\rf\leq n$. At each step, the set \( I_j \) contains at least one element of \( S \), and the corresponding element \( s_j \in S \cap I_j \) satisfies
\[
s_j > \sum_{i=0}^{j-1} s_i \text{ for } j=1,\dots, k.
\]
This yields $k+1$ elements $s_{0}, \dots, s_{k} \in S$ such that
$$\quad s_{j} > \sum_{i=0}^{j-1} s_{i}, \ \ \ \mbox{ for every }j=1,\dots, k,$$
as requested.
\end{proof}

We are now ready to prove that it is not possible to improve upon Construction \ref{constr:casenleqm}, that is, the codes defined as in Construction \ref{constr:casenleqm} exhibit the maximum possible number of non-zero weights for an $\Fqm$-linear rank metric code with the same parameters.

\begin{theorem} \label{th:upperboundn<m}
Assume that $n \leq m$. Let $s = \max\left\{\floor*{\frac{n}{2^{k-1}}}, 1\right\}$. Then
\[ L_{\rk}(n, m, k, q) \leq n - s+ 1.\]
\end{theorem}

\begin{proof}
If \( n < 2^k \), then \( s = 1 \), and the assertion becomes \( L_{\rk}(n, m, k, q) \leq n \), which is trivially true. So we may assume \( n \geq 2^k \), implying \( s \geq 2 \).
Suppose, by contradiction, that there exists an $[n,k]_{q^m/q}$ code \( \C \) with weight spectrum \( S \) such that \( |S| \geq n - s + 2 \).
By Lemma~\ref{lem:combinatoriallemma}, there exist \( k+1 \) elements \( s_0, \dots, s_k \in S \) such that
\begin{equation} \label{eq:relationsweightcombinatorics}
s_j > \sum_{i=0}^{j-1} s_i \quad \text{for every } 1 \leq j \leq k.
\end{equation}
Let \( \cc_0, \dots, \cc_k \in \C \) be codewords such that \( \w(\cc_j) = s_j \) for \( 0 \leq j \leq k \).
We claim that the codewords \( \cc_0, \dots, \cc_k \) are linearly independent. Indeed, suppose for some \( 1 \leq j \leq k \), we have a linear dependence of the form
\[
\cc_j = \sum_{i=0}^{j-1} \alpha_i \cc_i, \quad \alpha_i \in \F_{q^m}.
\]
Then, by the subadditivity of the rank weight under addition,
\[
\w(\cc_j) \leq \sum_{i=0}^{j-1} \w(\cc_i) = \sum_{i=0}^{j-1} s_i,
\]
which contradicts \eqref{eq:relationsweightcombinatorics}, since \( \w(\cc_j) = s_j > \sum_{i=0}^{j-1} s_i \). Therefore $\cc_{j}$ is not a linear combination of $\cc_{0}, \dots, \cc_{j-1}$, and the vectors $\cc_{0}, \dots, \cc_{j}$ are linearly independent for any $1 \leq j \leq k$.

Thus, we have \( k+1 \) linearly independent codewords in \( \C \), contradicting the assumption that $\C$ has dimension $k$. This proves the result.
\end{proof}

Combining Theorem~\ref{thm:constrnleqm} and Theorem~\ref{th:upperboundn<m}, we obtain Theorem~\ref{thm:small_n}. In particular, we completely characterize the values of \( n, m, k \) for which full weight spectrum (FWS) codes exist.

\begin{proposition} \label{prop:characterFWSn<m}
Let \( n \leq m \). A FWS \( [n, k]_{q^m/q} \) code exists if and only if \( n < 2^{k} \). In this case, the code \( \mathcal{C}_{\lambda,\mathbf{n}} \)  defined in Construction~\ref{constr:casenleqm} is a FWS code.
\end{proposition}

Note that, since the weight spectrum of an FWS code is \( \{1, \ldots, n\} \), such a code must have minimum distance 1. This situation arises in Construction~\ref{constr:casenleqm} either when \( z > 0 \), or when \( z = 0 \) and \( \left\lfloor \frac{n}{2^{k-1}} \right\rfloor \geq 1 \), with the additional condition that \( n < 2^{k} \). We now provide illustrative examples of FWS code constructions in both cases.

\begin{example}
Let \( n = 7 \), \( m = 7 \), and \( k = 3 \). Since \( n = 7 < 2^k = 8 \), by Proposition \ref{prop:characterFWSn<m}, we know that a full weight spectrum (FWS) code exists, and Construction~\ref{constr:casenleqm} provides an explicit construction. In this case, we have \( z = 0 \), and Construction~\ref{constr:casenleqm} gives the code \( \C_{\lambda,\mathbf{n}} \) with \( \mathbf{n} = (n_1, n_2, n_3) = (4, 2, 1) \), and generator matrix
\[
G_{\lambda, \mathbf{n}} = 
\begin{bmatrix}
\uu_{\lambda, 4} & {\bf 0} & {\bf 0} \\
{\bf 0} & \uu_{\lambda, 2} & {\bf 0} \\
{\bf 0} & {\bf 0} & \uu_{\lambda, 1}
\end{bmatrix}.
\]
This code has codewords of every weight from 1 to 7. Therefore, \( \C \) is an FWS \( [7, 3]_{q^{7}/q} \) code. 
\medskip

\noindent We now consider the case \( n = 7 \), \( m = 7 \), and \( k = 4 \). Since \( n = 7 < 2^k = 16 \), again by Proposition \ref{prop:characterFWSn<m}, we know that a full weight spectrum (FWS) code exists, and Construction~\ref{constr:casenleqm} provides an explicit construction. In this case, we have \( z = 1 > 0 \). Construction~\ref{constr:casenleqm} gives the code \( \C_{\lambda,\mathbf{n}} \) with \( \mathbf{n} = (n_1, n_2, n_3, n_4) = (3, 2, 1, 1) \), and generator matrix
\[
G_{\lambda, \mathbf{n}} = 
\begin{bmatrix}
\uu_{\lambda, 3} & {\bf 0} & {\bf 0} & {\bf 0} \\
{\bf 0} & \uu_{\lambda, 2} & {\bf 0} & {\bf 0} \\
{\bf 0} & {\bf 0} & \uu_{\lambda, 1} & {\bf 0} \\
{\bf 0} & {\bf 0} & {\bf 0} & \uu_{\lambda, 1}
\end{bmatrix}.
\]

The code \( \C_{\lambda,\mathbf{n}} \) contains codewords of every rank weight from \( 1 \) to \( 7 \), and therefore is an FWS \( [7,4]_{q^7/q} \) code.

$\hfill \lozenge$
\end{example}

We end this section with the following remark regarding the parameter $s$ introduced in Theorem \ref{thm:small_n}.

\begin{remark}
\label{rmk:nsmallandsgeqd}
Assume that $n \leq m$. Let $s = \max\left\{\floor*{\frac{n}{2^{k-1}}}, 1\right\}$. We remark that $s \geq d$ where $d$ is the minimum distance of an $\nkdqm$ code, with $\lvert \mathrm{WS}(\C)\rvert=L_{\rk}(n,m,k,q)$. By Theorem \ref{thm:small_n}, we have $L_{\rk}(n,m,k,q)=n-s+1\leq n$. Since $d$ is the minimum distance, we can have nonzero weights with values from $d$ up to $n$ and so at most $n-d+1$. Since $L_{\rk}(n,m,k,q)=n-s+1$, it follows that \[
 n-s+1\leq n-d+1,\]
and so $s\geq d$.
\end{remark}

\section{The case when \texorpdfstring{$n > m$}{n>m}}
\label{sec:n>m}

We now turn our attention to the case \( n > m \). First, recall that, without loss of generality, we may restrict our study to nondegenerate codes. Indeed, any degenerate code can be isometrically embedded into a smaller ambient space where it becomes nondegenerate, and the analysis can be carried out there, cf \cite[Section 3.1]{alfarano2022linear}. We also note that the following bound on \( n \) applies to nondegenerate codes. 

\begin{proposition} [see \textnormal{\cite[Corollary 6.5]{jurrius2017defining}}]
\label{prop:nleqkm}
Let $\C$ be a nondegenerate $[n,k]_{q^m/q}$ code. Then $n\leq km$.   
\end{proposition}

If $\C$ is a nondegenerate code and $n=km$, then $\C$ is a one-weight code, that is called a \textbf{simplex rank-metric code}, see \cite{alfarano2022linear,Randrianarisoa2020ageometric}.\\
We define the \textbf{simplex defect} of a nondegenerate $[n,k]_{q^m/q}$ code $\mathcal{C}$ the parameter \[\mu:=km-n,\] which represents how far is $\mathcal{C}$ from being a simplex rank-metric code.\\
Then define \[t:=\floor*{\frac{\mu}{m}}=\floor*{\frac{km-n}{m}}=k-\ceil*{\frac{n}{m}}\]
and
\begin{equation}
    \label{eq:a}
    a:=n-(k-t-2)m.
\end{equation}
Note that 
\begin{equation}
\label{eq:boundsonaandtfirst}
0\leq t\leq k-2 \ \ \ \mbox{and} \ \ \ n-(k-2)m\leq a\leq n.\end{equation}

In this section, we provide the following result, which corresponds to the case \( n > m \) in the Main Theorem.

\begin{theorem}
\label{thm:big_n}
Assume \( n > m \). Let \( h = \max\left\{ \left\lfloor \frac{a}{2^{t+1}} \right\rfloor, 1 \right\} \). Then
\[
L_{\rk}(m, n, k, q) = m - h + 1.
\]
\end{theorem}

To prove this result, we rely on the geometric description of linear rank-metric codes via systems, as developed in~\cite{Randrianarisoa2020ageometric,alfarano2022linear}, and we recall the notion of the dual of such spaces.

\subsection{The geometry of rank-metric codes and dual of a subspace}

Based on the classification of $\F_{q^m}$-linear isometries of the metric space $\F_{q^m}^n$ endowed with the rank distance, two $[n,k]_{q^m/q}$ codes $\C$ and $\C'$ are said to be \textbf{equivalent} if and only if there exists a matrix \( A \in \mathrm{GL}(n, q) \) such that
\[
\C' = \C A = \{ \vv A : \vv \in \C \}.
\]

The geometric counterpart of rank-metric codes is captured by the notion of \emph{systems}. An $\nkdqm$ (or simply $[n,k]_{q^m/q}$) \textbf{system} \( U \) is an $\F_q$-subspace of \( \F_{q^m}^k \) of dimension \( n \) such that \( \langle U \rangle_{\F_{q^m}} = \F_{q^m}^k \). The parameter \( d \) is defined by
\[
d = n - \max\left\{ \dim_{\F_q}(U \cap H) \,\middle|\, H \text{ is an } \F_{q^m}\text{-hyperplane of } \F_{q^m}^k \right\}.
\]

Two $\nkdqm$ systems \( U \) and \( U' \) are said to be \textbf{equivalent} if there exists an invertible matrix \( B \in \mathrm{GL}(k, \F_{q^m}) \) such that
\[
U' = U \cdot B = \{ \uu B \mid \uu \in U \}.
\]

The correspondence between rank-metric codes and systems is as follows. Given a nondegenerate $\nkdqm$ code \( \C \) with generator matrix \( G \), the $\F_q$-span of the columns of \( G \) defines an $\nkdqm$ system \( U \), referred to as a \textbf{system associated with} \( \C \).\\
Conversely, if \( U \) is an $\nkdqm$ system, then one may construct a generator matrix \( G \) by choosing an $\F_q$-basis of \( U \) and placing the basis vectors as columns. The resulting code generated by \( G \) is an $\nkdqm$ code, called a \textbf{code associated with} \( U \).

Let \( \sigma \) denote the standard inner product on \( \F_{q^m}^k \), i.e.
\[
\sigma\colon \F_{q^m}^k \times \F_{q^m}^k \to \F_{q^m}, \quad \sigma\left((x_1,\ldots,x_k), (y_1,\ldots,y_k)\right) = \sum_{i=1}^k x_i y_i.
\]
For a vector \( \xx = (x_1, \ldots, x_k) \in \F_{q^m}^k \), its orthogonal complement with respect to \( \sigma \) is given by
\[
\xx^\perp = \left\{ \mathbf{y} = (y_1, \ldots, y_k) \in \F_{q^m}^k \,\middle|\, \sigma(\xx, \mathbf{y}) = 0 \right\}.
\]

The following result establishes the key connection between the rank weight of a codeword and the intersection of a hyperplane with the system associated with the code (see also \cite{alfarano2022linear,neri2021geometry}).

\begin{theorem}[see \textnormal{\cite{Randrianarisoa2020ageometric}}]
\label{th:connection}
Let \( \C \) be a nondegenerate $[n,k]_{q^m/q}$ code. Let $G$ be a generator matrix of $\C$, and let \( U \subseteq \F_{q^m}^k \) be the system associated with $\C$, defined as the $\F_q$-span of the columns of \( G \). Then, for any \( \xx \in \F_{q^m}^k \), the weight of the codeword \( \xx G \in \C \) is given by
\begin{equation}\label{eq:relweight}
\w(\xx G) = n - \dim_{\F_q}(U \cap \xx^\perp).
\end{equation}
\end{theorem}

Systems associated with equivalent codes are themselves equivalent, and the converse also holds. For further details, see~\cite{alfarano2022linear}.

We now recall the concept of the dual of an $\F_q$-subspace of $\F_{q^m}^k$.

Let $\mathrm{Tr}_{q^m/q}(x) = x + x^q + \ldots + x^{q^{m-1}}$ denote the trace function from $\F_{q^m}$ to $\F_q$. Define the $\F_q$-bilinear form
\[
\begin{array}{cccc}
    \sigma' \colon & \F_{q^m}^k \times \F_{q^m}^k & \longrightarrow & \F_q \\
    & (\xx, \mathbf{y}) & \longmapsto & \mathrm{Tr}_{q^m/q}(\sigma(\xx, \mathbf{y})),
\end{array}
\]
where \( \sigma \) is the standard inner product on \( \F_{q^m}^k \). The form \( \sigma' \) is a nondegenerate, reflexive bilinear form on \( \F_{q^m}^k \), viewed as an $\F_q$-vector space of dimension \( km \). Let \( \perp' \) denote the orthogonal complement map with respect to \( \sigma' \). For any $\F_q$-subspace \( U \subseteq \F_{q^m}^k \), the \textbf{dual subspace} of \( U \) is defined as
\[
U^{\perp'} = \left\{ \mathbf{v} \in \F_{q^m}^k \,\middle|\, \sigma'(\uu, \mathbf{v}) = 0 \text{ for all } \uu \in U \right\}.
\]
Clearly, \( \dim_{\F_q}(U^{\perp'}) = km - \dim_{\F_q}(U) \). Moreover, for any $\F_{q^m}$-subspace \( W \subseteq \F_{q^m}^k \), we have \( W^{\perp'} = W^\perp \) (see \cite{polverino2010linear}). As a consequence, we have the following result.

\begin{proposition}[see~\textnormal{\cite[Property 2.6]{polverino2010linear}}]
\label{prop:weightdual}
Let \( U \) be an $\F_q$-subspace and \( W \) an $\F_{q^m}$-subspace of \( \F_{q^m}^k \). Then:
\[
\dim_{\F_q}(U^{\perp'} \cap W^\perp) = \dim_{\F_q}(U \cap W) + km - \dim_{\F_q}(U) - \dim_{\F_q}(W).
\]
\end{proposition}

The next result connects the rank weight of a codeword with the geometry of the associated system via orthogonal complement, see e.g. \cite[Proposition 2.11]{santonastaso2024completely}.

\begin{proposition} \label{prop:characweightgeometricdual}
Let \( \C \) be a nondegenerate $[n,k]_{q^m/q}$ code. Let $G$ be a generator matrix of $\C$, and let \( U \subseteq \F_{q^m}^k \) be the system associated with $\C$, defined as the $\F_q$-span of the columns of \( G \). Then for every \( \xx \in \F_{q^m}^k \), the following equivalences hold:
\begin{enumerate}
    \item $\w(\xx G) = i$;
    \item $\dim_{\F_q}(U \cap \xx^\perp) = n - i$;
    \item $\dim_{\F_q}(U^{\perp'} \cap \langle \xx \rangle_{\F_{q^m}}) = m - i$.
\end{enumerate}
\end{proposition}

\subsection{Proof of Theorem \ref{thm:big_n}}

As for the case $n\leq m$, we will establish Theorem \ref{thm:big_n} by showing that $m-h+1$ is both an upper and a lower bound on $L_{\rk}(n,m,k,q)$.

We start with the following proposition which gives a better estimation of the parameter $a$ with respect to Equation \ref{eq:boundsonaandtfirst}.

\begin{proposition}
\label{prop:m<a leq 2m}
Assume that $n>m$. The parameter $a$ defined as in \eqref{eq:a} satisfies
\[m+1 \leq a \leq 2m.\]
\end{proposition}

\begin{proof}
First, note that \[a = n - (k-t-2)m = tm - \mu + 2m.\]

Since $t=\floor*{\frac{\mu}{m}}$, we can write $\mu = tm + r$, where $r \in \{0\, \dots, m-1\}$ is the residue of $\mu$ modulo $m$. Then it is clear that $tm - \mu = -r$, which yields
$$a = 2m-r.$$
This completes the proof.
\end{proof}

We provide an explicit construction of a code with exactly $m-h+1$ nonzero weights. This construction will differ according to whether $a$ is greater or lower than $t+2$. In the first case, we will involve again the function $\Psi(u,v)$ of Definition \ref{def:definitionPsi} to define suitably the vector $\bf{n}$ in \eqref{eq:generatormatrixlambdageneral}. In the second case, we will use a different approach.\\

As proven in Theorem \ref{thm:constrn>m} below, this construction will also imply that the value $m-h+1$ is a lower bound on $L_{\rk}(m,n,k,q)$.

\begin{construction}
\label{constr:casen>m}
Assume \( n > m \), and recall that \( a = n - (k - t - 2)m \) and \( t = \left\lfloor \frac{km - n}{m} \right\rfloor \).
First assume that $a \geq t+2$.
Define the set
\[
H := \left\{ i \in \Z_{\geq 0} \colon \frac{a - i}{2^{t+1 - i}} \geq 1 \right\}.
\]
Note that since $a\geq t+2$, we have \( t+1 \in H \), and therefore the set $H$ is nonempty.
Let \( z := \min H \leq t + 1 \).

\medskip
\noindent
\textbf{\underline{Case $a \geq t+2$ and \( z = 0 \).}} Consider the tuple \( \mathbf{n'} := \Psi(a, t+2) = (m_1, \ldots, m_{t+2}) \), where the function \( \Psi \) is defined in Definition~\ref{def:definitionPsi}. Set the tuple
\begin{equation} \label{eq:partitionn>m}
\mathbf{n} = (\underbrace{m, \ldots, m}_{k - t - 2\ \text{times}}, m_1, \ldots, m_{t+2}).
\end{equation}
Note that $m_{t+2} = \lfloor a/2^{t+1} \rfloor \geq 1$,  \( \sum_{i=1}^{t+2} m_i = a \) and
\( n = a + m(k - t - 2) \). Thus, we define the \( [n,k]_{q^m/q} \) code 
\( \mathcal{C}_{\lambda, \mathbf{n}} \), with generator matrix 
\( G_{\lambda, \mathbf{n}} \) as in~\eqref{eq:generatormatrixlambdageneral}, that has the desired length \( n \).

\medskip
\noindent
\textbf{\underline{Case $a \geq t+2$ and \( z > 0 \).}} Let \( a' := a - z \) and \( t' := t - z \). Consider the tuple \( \Psi(a', t'+2) = (m_1, \ldots, m_{t'+2}) \), and define the extended tuple
\[
\mathbf{n} := (\underbrace{m, \ldots, m}_{k - t - 2\ \text{times}}, m_1, \ldots, m_{t'+2}, \underbrace{1, \ldots, 1}_{z\ \text{times}}).
\]
Again, note that $m_{t'+2} = \lfloor \frac{a-z}{2^{t+1-z}} \rfloor \geq 1$ and
that \( n = (k-t-2)m+\sum_{i=1}^{t'+2}m_i+z\). We thus define the \( [n,k]_{q^m/q} \) code 
\( \mathcal{C}_{\lambda, \mathbf{n}} \) with generator matrix 
\( G_{\lambda, \mathbf{n}} \) as in~\eqref{eq:generatormatrixlambdageneral}, that has the desired length $n$. 

\medskip

\medskip
\noindent
\textbf{\underline{Case $a < t+2$.}} If we try to make the previous construction work, our problem will be that $a$ is too small to fill the remaining $t+2$ rows. We solve this problem by taking some of the first $k-t-2$ rows to fill the last rows with $1$'s.

Write $\beta = \lf \frac{n-k}{m-1} \rf$ and $\gamma = (n-k) - (m-1)\beta$. Note that $\gamma$ is the remainder of the division of $n-k$ by $m-1$, hence $\gamma \in \{0, \dots, m-2\}$.

\noindent Consider the tuple
\[
\mathbf{n} := (\underbrace{m, \ldots, m}_{\beta\ \text{times}}, \gamma+1, \underbrace{1, \ldots, 1}_{k-\beta-1\ \text{times}}).
\]
First note that there are $k$ coefficients in this tuple, and they are all between $1$ and $m$. Next, note that
\begin{align*}
\beta m + \gamma+1 + k-\beta - 1 &= \beta (m-1) + \gamma+1 + k - 1 \\
&= n-k + k \\
&= n.
\end{align*}

We then define the \( [n,k]_{q^m/q} \) code 
\( \mathcal{C}_{\lambda, \mathbf{n}} \) with generator matrix 
\( G_{\lambda, \mathbf{n}} \) as in~\eqref{eq:generatormatrixlambdageneral}, that has the desired length $n$. 
\end{construction}

We now prove that the code $\C_{\lambda, \bf{n}}$ defined above has exactly $m-h+1$ distinct weights.

\begin{theorem}
\label{thm:constrn>m}
Let \( n > m \), and define
\[
h = \max\left\{ \left\lfloor \frac{a}{2^{t+1}} \right\rfloor,\, 1 \right\},
\]
where \( a = n - (k - t - 2)m \) and \( t = \left\lfloor \frac{km - n}{m} \right\rfloor \).  
Then, the \( [n,k]_{q^m/q} \) code \( \mathcal{C}_{\lambda, \mathbf{n}} \), defined as in Construction~\ref{constr:casen>m}, has weight spectrum
\[
\mathrm{WS}(\mathcal{C}_{\lambda,\mathbf{n}}) =
\{h, h+1, \ldots, m\}
\]
and thus attains exactly \( m - h + 1 \) distinct nonzero weights.
\end{theorem}

\begin{proof}

We shall treat the cases presented in Construction \ref{constr:casen>m} in order.  

\medskip
\noindent
\textbf{\underline{Case $a \geq t+2$ and $z=0$.}}

For every \( i = 1, \dots, t+2 \), let \( \mathbf{c}_{m_i} = (\mathbf{0}, \ldots, \uu_{\lambda,m_i}, \ldots, \mathbf{0}) \in \mathcal{C}_{\lambda,\mathbf{n}} \), where $\mathbf{c}_{m_i}$ is written as a vector defined by blocks of length $m, m_1,m_2,\dots,m_{t+2}$ and \( \uu_{\lambda,m_i} \) appears in the block of length $m_i$, for $i=1,\dots, t+2$. Note that \( m_{t+2} \leq m_{t+1} \leq \dots \leq m_1 \leq m \), and the codeword \( \mathbf{c}_{m_{t+2}} = (\mathbf{0}, \ldots, \mathbf{0}, \uu_{\lambda,m_{t+2}}) \in \mathcal{C}_{\lambda,\mathbf{n}} \) has rank weight $$\w(\mathbf{c}_{m_{t+2}}) = \dim_{\F_q}(\langle \uu_{\lambda,m_{t+2}} \rangle_{\F_q}) = \dim_{\F_q}(\langle 1, \lambda, \ldots, \lambda^{m_{t+2} - 1} \rangle_{\F_q}) = m_{t+2},$$
which is the minimum possible.

Similarly to the proof of Theorem \ref{thm:constrnleqm}, we show how to generate codewords covering all weights from \( m_{t+2} \) up to \( m \). For every \( j = 0, \ldots, m_{t+2} \), consider:
\[
\mathbf{c}_{m_{t+2}} + \lambda^j \mathbf{c}_{m_{t+1}} = (\mathbf{0}, \ldots, \lambda^j \uu_{\lambda,m_{t+1}}, \uu_{\lambda,m_{t+2}}) \in \mathcal{C}_{\lambda,\mathbf{n}}.
\]
Then, since $m_{t+2}\leq m_{t+1}$, we have that 
\[
\w(\mathbf{c}_{m_{t+2}} + \lambda^j \mathbf{c}_{m_{t+1}}) = \dim_{\F_q}(\langle 1, \lambda, \ldots,\lambda^{m_{t+1}},\ldots, \lambda^{m_{t+1} + j - 1} \rangle_{\F_q}) = m_{t+1} + j.
\]
By Proposition~\ref{prop:propertiesofPSI}(ii), we know that either
\[
m_{t+1} = m_{t+2} \quad \text{or} \quad m_{t+1} = m_{t+2} + 1.
\]
So, by letting \( j \) vary from \( 0 \) to \( m_{t+2} \), we obtain codewords of weights \(  m_{t+2} + 1, \ldots, m_{t+1} + m_{t+2} \). Again by Proposition~\ref{prop:propertiesofPSI}(ii), we have
\[
\text{ either } m_{t+1} + m_{t+2}=m_{t}  \quad \text{or} \quad m_{t+1} + m_{t+2}= m_{t}-1.
\]
Therefore, we cover all weights up to either $m_{t}-1$ or $m_{t}$.

\noindent Now, for each \( j = 0, \ldots, m_{t+1} + m_{t+2} \), consider:
\[
\mathbf{c}_{m_{t+2}} + \lambda^{m_{t+2}} \mathbf{c}_{m_{t+1}} + \lambda^j \mathbf{c}_{m_{t}} = (\mathbf{0}, \ldots, \lambda^j \uu_{\lambda,m_{t}}, \lambda^{m_{t+2}} \uu_{\lambda,m_{t+1}}, \uu_{\lambda,m_{t+2}}) \in \mathcal{C}_{\lambda,\mathbf{n}}.
\]
Then,
\[
\begin{aligned}
\w(\mathbf{c}_{m_{t+2}} + \lambda^{m_{t+2}} \mathbf{c}_{m_{t+1}} + \lambda^j \mathbf{c}_{m_{t}}) &= \dim_{\F_q}(\langle 1, \lambda, \ldots,\lambda^{m_k-1},\lambda^{m_k}, \dots, \lambda^{m_{t+2}+m_{t+1}-1},\dots, \lambda^j, \dots, \lambda^{m_{t} + j - 1} \rangle_{\F_q})\\ &= m_{t} + j.
\end{aligned}
\]
So, we find codewords of weights \( m_{t}, \ldots, m_{t} + m_{t+1} + m_{t+2} \). Again by Proposition~\ref{prop:propertiesofPSI}(ii), we know that this covers weights up to \( m_{t-1} \) or \( m_{t-1} - 1 \) since \[
\text{ either } m_{t} + m_{t+1} + m_{t+2}=m_{t-1}-1 \text{ or } m_{t} + m_{t+1} + m_{t+2}=m_{t-1}.\]

Proceeding recursively in this way, after a finite number of steps, we obtain codewords of every weight from \( m_{t+2} \) to \( m_2+m_3+\dots+m_{t+2} \), where:
\[
\text{ either } m_2+\dots+m_{t+2} = m_1 - 1 \leq m
\text{ or } m_2+\dots+m_{t+2} = m_1 \leq m.
\]

Finally, for every \( j = 0, \ldots, m_2 + \cdots + m_{t+2} \), consider:
\[
\mathbf{c}_{m_{t+2}} + \lambda^{m_{t+2}} \mathbf{c}_{m_{t+1}} + \lambda^{m_{t+1}} \mathbf{c}_{m_{t}} + \cdots + \lambda^{m_3} \mathbf{c}_{m_2} + \lambda^j \mathbf{c}_{m_1} \in \mathcal{C}_{\lambda, \mathbf{n}}.
\]
Then the weight is:
\[
\begin{aligned}
&\w(\mathbf{c}_{m_{t+2}} + \lambda^{m_{t+2}} \mathbf{c}_{m_{t+1}} + \lambda^{m_{t+1}} \mathbf{c}_{m_{t}} + \cdots + \lambda^{m_3} \mathbf{c}_{m_2} + \lambda^j \mathbf{c}_{m_1})=\\
&=\dim_{\F_q}(\langle 1, \lambda, \ldots,\lambda^{m_{t+2}-1},\lambda^{m_{t+2}},\dots, \lambda^{m_{t+1}-1}, \lambda^{m_{t+1}}, \dots,\lambda^j,\dots, \lambda^{m_1 + j - 1} \rangle_{\F_q}) \\
&=m_1 + j.
\end{aligned}
\]

Therefore, by letting \( j \) vary, we obtain codewords whose weights range from 
\( m_1 \) up to \( m \), with \( m \le a = m_1 + \cdots + m_{t+2} \), as guaranteed by (i) of Proposition \ref{prop:propertiesofPSI}.

\medskip
\noindent
\textbf{\underline{Case $a \geq t+2$ and \( z > 0 \).}} 
We first observe that $m_{t'+2} = \lfloor \frac{a-z}{2^{t+1-z}} \rfloor = 1$. Assume indeed, by contradiction, that
\[
\left\lfloor \frac{a - z}{2^{t +1- z}} \right\rfloor \ge 2 .
\]
Then
\[
\frac{a - z}{2^{t + 1-z}} \ge 2,
\]
which implies
\[
\frac{a - z + 1}{2^{t + 1-z}} \ge 2.
\]
Hence
\[
\frac{a - z + 1}{2^{t +2- z}} \ge 1,
\]
and so
\[
\left\lfloor \frac{a - z + 1}{2^{t +2- z}} \right\rfloor \ge 1.
\]
This shows that \( z - 1 \in H\), where \(H\) is defined in
Construction \ref{constr:casen>m}, contradicting the minimality of \(z\).
\\
Let \(\overline{n} = a' = a - z\) and \(\overline{k} = t' + 2 = t - z + 2\).
By our assumptions, we have \(\overline{n} \ge \overline{k}\), and since
\(a \le 2m\) (see Proposition~\ref{prop:m<a leq 2m}), it follows that
\(\overline{n} < 2m\).
Moreover,
\[
\Psi(\overline{n}, \overline{k}) = \Psi(a', t' + 2)
= (m_1, \ldots, m_{t'+2}),
\]
and we consider the \([\overline{n}, \overline{k}]_{q^m/q}\) code
\[
\overline{\mathcal{C}} =
\left\{
(\alpha_{m_1} \uu_{\lambda,m_1}, \ldots, \alpha_{\lambda,m_{t'+2}} \uu_{\lambda,m_{t'+2}})
\;\middle|\;
\alpha_{m_i} \in \F_{q^m},\ i = 1, \ldots, t' + 2
\right\}.
\]

Assume first that \(\overline{n} > m\).
For the parameters of the code \(\overline{\mathcal{C}}\), we then have
\(\overline{t} = \overline{k} - 2 = t'\) and
\[
\overline{a}
= \overline{n} - (\overline{k} - \overline{t} - 2)m
= \overline{n}=a'.
\]
Therefore, the code \(\overline{\mathcal{C}}\) coincides with the code constructed 
in the previous case and has \(m - \overline{h} + 1\) 
distinct weights, where
\[
\overline{h}
= \left\lfloor \frac{\overline{a}}{2^{\overline{t}+1}} \right\rfloor
= \left\lfloor \frac{a - z}{2^{t - z + 1}} \right\rfloor=m_{t'+2}=1.
\]
Consequently, we obtain $\lvert \mathrm{WS}(\overline{\C})\rvert=m$, and so
\[
\mathrm{WS}(\overline{\mathcal{C}}) = \{1, \ldots, m\}.
\]
Furthermore, every codeword \( \overline{\mathbf{c}} \in \overline{\mathcal{C}} \) corresponds to the codeword \( (\mathbf{0}, \overline{\mathbf{c}}, \mathbf{0}) \in \mathcal{C} \), so we conclude that
\[
\{1, \ldots, m\} = \mathrm{WS}(\overline{\mathcal{C}}) = \mathrm{WS}(\mathcal{C}).
\]

Assume now that \(\overline{n} \le m\).
The \([\overline{n}, \overline{k}]_{q^m/q}\) code
\(\overline{\mathcal{C}}\) is exactly the code obtained from
Construction \ref{constr:casenleqm}, with
\[
z = \min\left\{
i \in \mathbb{Z}_{\ge 0} \;\middle|\;
\frac{\overline{n} - i}{2^{\overline{k} - i - 1}} \ge 1
\right\}
=
\min\left\{
i \in \mathbb{Z}_{\ge 0} \;\middle|\;
\frac{a - z - i}{2^{t-z + 1 - i}} \ge 1
\right\}
= 0.
\]
Therefore, by Theorem \ref{thm:constrnleqm}, we obtain
\[
\mathrm{WS}(\overline{\mathcal{C}})
=
\{\overline{s}, \overline{s}+1, \ldots, \overline{n}\},
\]
where
\[
\overline{s}
=
\left\lfloor \frac{\overline{n}}{2^{\overline{k}-1}} \right\rfloor
=
m_{t'+2}
= 1.
\]
Hence,
$
\lvert \mathrm{WS}(\overline{\mathcal{C}}) \rvert = \overline{n} = a - z,
$
and
$
\mathrm{WS}(\overline{\mathcal{C}}) = \{1, \ldots, a - z\}.$ Again, every codeword
\(\overline{\mathbf{c}} \in \overline{\mathcal{C}}\)
corresponds to the codeword
\((\mathbf{0}, \overline{\mathbf{c}}, \mathbf{0}) \in \mathcal{C}\),
and hence
\[
\{1, \ldots, a - z\}
=
\mathrm{WS}(\overline{\mathcal{C}})
\subseteq
\mathrm{WS}(\mathcal{C}).
\]

It remains to show that there exist codewords in \(\mathcal{C}\) of weights
\(a - z + 1, \ldots, m\).
By construction, there exists
\(\overline{\mathbf{c}}' \in \overline{\mathcal{C}}\) such that
\[
\langle \overline{\mathbf{c}}' \rangle_{\mathbb{F}_q}
=
\langle 1, \lambda, \ldots, \lambda^{a - z - 1} \rangle_{\mathbb{F}_q},
\qquad
\mathrm{w}(\overline{\mathbf{c}}') = a - z.
\]

Now consider the family of codewords of the form
\[
(\mathbf{0}, \overline{\mathbf{c}}', \lambda^{i_1}, \ldots, \lambda^{i_z})
\in \mathcal{C},
\]
for arbitrary non-negative integers \(i_1, \ldots, i_z\).
By choosing the tuples
\begin{align*}
(i_1, \ldots, i_z) &= (a - z, 0, \ldots, 0), \\
(i_1, \ldots, i_z) &= (a - z, a - z + 1, 0, \ldots, 0), \\
&\vdots \\
(i_1, \ldots, i_z) &= (a - z, a - z + 1, \ldots, a - 1, 0), \\
(i_1, \ldots, i_z) &= (a - z, a - z + 1, \ldots, a - 1, a),
\end{align*}
we obtain codewords with weights ranging from \(a - z + 1\) up to
\(\min\{a,m\} = m\). Therefore,
\[
\mathrm{WS}(\mathcal{C}) = \{1, \ldots, m\},
\]
and the assertion follows.

\medskip
\noindent
\textbf{\underline{Case $a < t+2$.}}
This yields $h=1$, so we must show that all nonzero weights can be attained. As shown in Construction \ref{constr:casen>m}, there are $k - \beta -1$ rows of weight $1$. The generator matrix $G_{\lambda, {\bf n}}$ therefore has the form
\begin{align*}
G_{\lambda, \mathbf{n}} &= 
\begin{bmatrix}
\uu_{\lambda, m} & {\bf 0} & {\bf 0} & {\bf 0} & {\bf 0} &{\bf 0} &{\bf 0} \\
{\bf 0} & \ddots & {\bf 0} & {\bf 0} & {\bf 0} & {\bf 0} & {\bf 0} \\
{\bf 0} & {\bf 0} & \uu_{\lambda, m} & {\bf 0} & {\bf 0} & {\bf 0} & {\bf 0} \\
{\bf 0} & {\bf 0} & {\bf 0} & \uu_{\lambda, 1+\gamma} & {\bf 0} & {\bf 0} & {\bf 0} \\
{\bf 0} & {\bf 0} & {\bf 0} & {\bf 0} & \uu_{\lambda, 1} & {\bf 0} & {\bf 0} \\
{\bf 0} & {\bf 0} & {\bf 0} & {\bf 0} & {\bf 0} & \ddots & {\bf 0} \\
{\bf 0} & {\bf 0} & {\bf 0} & {\bf 0} & {\bf 0} & {\bf 0} & \uu_{\lambda, 1} \\
\end{bmatrix} \\
&= \begin{bmatrix}
\uu_{\lambda, m} & {\bf 0} & {\bf 0} & {\bf 0} & {\bf 0} &{\bf 0} &{\bf 0} \\
{\bf 0} & \ddots & {\bf 0} & {\bf 0} & {\bf 0} & {\bf 0} & {\bf 0} \\
{\bf 0} & {\bf 0} & \uu_{\lambda, m} & {\bf 0} & {\bf 0} & {\bf 0} & {\bf 0} \\
{\bf 0} & {\bf 0} & {\bf 0} & \uu_{\lambda, 1+\gamma} & {\bf 0} & {\bf 0} & {\bf 0} \\
{\bf 0} & {\bf 0} & {\bf 0} & {\bf 0} & 1 & {\bf 0} & {\bf 0} \\
{\bf 0} & {\bf 0} & {\bf 0} & {\bf 0} & {\bf 0} & \ddots & {\bf 0} \\
{\bf 0} & {\bf 0} & {\bf 0} & {\bf 0} & {\bf 0} & {\bf 0} & 1 \\
\end{bmatrix}
\end{align*}

Now we prove that there are at least $m$ different coefficients equal to one, or in other words, that $k - \beta - 1 \geq m$. First of all, we have
$$a = n - (k-t-2)m < t+2,$$
which yields
$$m+1 \leq a < t+2 < k - \frac{n-k}{m-1}.$$
Simplifying, we get
$$m+3 \leq  k-\frac{n-k}{m-1},$$
and finally
$$m+3 \leq k - \beta.$$
This yields $m+2 \leq k - \beta -1$, and hence $m \leq k - \beta -1$, as claimed.

\smallskip

For every $j \in \{1, \dots, m\}$, consider the vector ${\bf a}_{j} = (0, \dots, 0, \underbrace{1, \lambda \dots, \lambda^{j-1}}_{j }) \in \F_{q^m}^k$. Then the codeword ${\bf a}_j \cdot G_{\lambda, \mathbf{n}}$ has rank weight exactly $j$. Note in particular that $m \leq k - \beta - 1$, as observed before, so that all possible weights from $1$ to $m$ can be constructed in this way.

\end{proof}

As a consequence of the above result, we get the following lower bound on $L_{\rk}(n,m,k,q)$. 

\begin{corollary}
\label{cor:lowerboundn>m}
Assume that $n > m$. Let $h = \max\left\{\floor*{\frac{a}{2^{t+1}}}, 1\right\}$. Then
\[ L_{\rk}(m, n, k, q) \geq m - h+ 1.\]    
\end{corollary}

In the next, we present some examples of codes defined as in Construction \ref{constr:casen>m} and we show that their weight spectrums have size $m-h+1$.

\begin{example}
Let \( n = 11 \), \( m = 5 \), and \( k = 4 \). In this case, $t=\lfloor\frac{km-n}{m}\rfloor=\lfloor\frac{20-11}{5}\rfloor=1$ and $a=n-(k-t-2)m=11-(4-1-2)5=6\geq t+2=3$. So we have that $h=\lfloor\frac{a}{2^{t+1}}\rfloor=1$ and, by \Cref{thm:constrn>m}, Construction~\ref{constr:casen>m} provides a code with $m-h+1=5$ non-zero weights. In this case, we have \( z = \min\lbrace i\geq0\colon \frac{6-i}{2^{t+1-i}}\geq 1\rbrace=0 \), and Construction~\ref{constr:casenleqm} gives the code \( \C_{\lambda,\mathbf{n}} \) with \( \mathbf{n} = (m, m_1, m_2, m_3) = (5, 3, 2, 1) \), and generator matrix
\[
G_{\lambda, \mathbf{n}} = 
\begin{bmatrix}
\uu_{\lambda, 5} & {\bf 0} & {\bf 0} & {\bf 0} \\
{\bf 0} & \uu_{\lambda, 3} & {\bf 0} & {\bf 0} \\
{\bf 0} & {\bf 0} & \uu_{\lambda, 2} & {\bf 0}\\
{\bf 0} & {\bf 0} & {\bf 0} & \uu_{\lambda, 1}
\end{bmatrix}.
\]
This code has codewords of every weight from 1 to 5. Indeed, following the proof of Theorem \ref{thm:constrn>m}, we have that 
\begin{itemize}
\item \( \mathbf{c}_{m_3}=(0,0,0,0,0,0,0,0,0,0, \uu_{\lambda, 1})=(0,0,0,0,0,0,0,0,0,0,1) \) has weight 1.
    \item $\mathbf{c}_{m_3} + \mathbf{c}_{m_2} = (0,0,0,0,0,0,0,0, \lambda\uu_{\lambda,2}, \uu_{\lambda,1})=(0,0,0,0,0,0,0,0,1,\lambda,1)$ has weight 2.
    \item $\mathbf{c}_{m_3} + \lambda\mathbf{c}_{m_2} = (0,0,0,0,0,0,0,0, \lambda\uu_{\lambda,2}, \uu_{\lambda,1})=(0,0,0,0,0,0,0,\lambda,\lambda^2,1)$ has weight 3.
    \item $\mathbf{c}_{m_3} + \lambda \mathbf{c}_{m_2}+\lambda\mathbf{c}_{m_1}=(0,0,0,0,0,\lambda\uu_{\lambda,3}, \lambda\uu_{\lambda,2}, \uu_{\lambda, 1})=(0,0,0,0,0,\lambda,\lambda^2, \lambda^3,\lambda,\lambda^2,1)$ has weight 4.
    \item $\mathbf{c}_{m}=(\mathbf{u}_{\lambda,5},0,0,0,0,0,0)=(1,\lambda,\lambda^2,\lambda^3,\lambda^4,0,0,0,0,0,0)$ has weight 5.
\end{itemize}
\medskip

\noindent Now, we consider the case \( n = 12 \), \( m = 7 \), and \( k = 6 \). In this case, $t=\lfloor\frac{km-n}{m}\rfloor=\lfloor\frac{42-12}{7}\rfloor=4$ and $a=n-(k-t-2)m=12-(6-4-2)7=12\geq t+2=6$. So we have that $h=\lfloor\frac{a}{2^{t+1}}\rfloor=1$ and, by \Cref{thm:constrn>m}, Construction~\ref{constr:casen>m} provides a code with $m-h+1=5$ non-zero weights. In this case, we have \( z = \min\lbrace i\geq0\colon \frac{12-i}{2^{t+1-i}}\geq 1\rbrace=2>0\), and Construction~\ref{constr:casen>m} gives the code \( \C_{\lambda,\mathbf{n}} \) with \( \mathbf{n} = (m_1, m_2, m_3, m_4,1,1) = (5,3,1,1,1,1) \), and generator matrix
\[
G_{\lambda, \mathbf{n}} = 
\begin{bmatrix}
\uu_{\lambda, 5} & {\bf 0} & {\bf 0} & {\bf 0} & {\bf 0} & {\bf 0}\\
{\bf 0} & \uu_{\lambda, 3} & {\bf 0} & {\bf 0} &{\bf 0} & {\bf 0}\\
{\bf 0} & {\bf 0} & \uu_{\lambda, 1} & {\bf 0} & {\bf 0}& {\bf 0}\\
{\bf 0} & {\bf 0} & {\bf 0} & \uu_{\lambda, 1} & {\bf 0}& {\bf 0}\\
{\bf 0} & {\bf 0} & {\bf 0}  & {\bf 0}& \uu_{\lambda, 1}& {\bf 0}\\
{\bf 0} & {\bf 0} & {\bf 0}  & {\bf 0} & {\bf 0}& \uu_{\lambda, 1}
\end{bmatrix}.
\]
This code has codewords of every weight from 1 to 7. Indeed, following the proof of Theorem \ref{thm:constrn>m}, consider the punctured code of $\mathcal{C}$, defined as
\[
\mathcal{C}'=\lbrace (\alpha_{m_1}{\bf u}_{m_1}, \alpha_{m_2}{\bf u}_{m_2}, \alpha_{m_3}{\bf u}_{m_3},\alpha_{m_4}{\bf u}_{m_4})\mid \alpha_{m_i}\in\F_{q^m},i=1,2,3,4\rbrace
\]
Then
\begin{itemize}
\item \( \mathbf{c}'_1=\mathbf{c}'_{m_4}=(0,0,0,0,0,0,0,0,0, \uu_{\lambda, 1})=(0,0,0,0,0,0,0,0,0,1) \) has weight 1.
    \item $\mathbf{c}'_2=\mathbf{c}'_{m_4} + \lambda\mathbf{c}'_{m_3} = (0,0,0,0,0,0,0,0, \lambda\uu_{\lambda,1}, \uu_{\lambda,1})=(0,0,0,0,0,0,0,0,\lambda,1)$ has weight 2.
    \item $\mathbf{c}'_3=\mathbf{c}'_{m_4} + \lambda\mathbf{c}'_{m_3}+ \mathbf{c}'_{m_2}= (0,0,0,0,0,0,\uu_{\lambda,3}, \lambda\uu_{\lambda,1}, \uu_{\lambda,1})=(0,0,0,0,0,1,\lambda,\lambda^2,\lambda,1)$ has weight 3.
    \item $\mathbf{c}'_4=\mathbf{c}'_{m_4} + \lambda \mathbf{c}'_{m_3}+\lambda\mathbf{c}'_{m_2}+\mathbf{c}'_{m_1}=(0,0,0,0,0,\lambda\uu_{\lambda,3}, \lambda\uu_{\lambda,1}, \uu_{\lambda, 1})=(0,0,0,0,0,\lambda,\lambda^2, \lambda^3,\lambda,1)$ has weight 4.
    \item $\mathbf{c}'_5=\mathbf{c}'_{m_1}=(\mathbf{u}_{\lambda,5},0,0,0,0,0)=(1,\lambda,\lambda^2,\lambda^3,\lambda^4,0,0,0,0,0)$ has weight 5.
\end{itemize}
Then the codewords $\mathbf{c}_j=(\mathbf{c}'_j,0,0)\in\mathcal{C}$ for every $j=1,\dots,5$ and so $\mathcal{C}$ has codewords of weight from 1 to 5. Finally, consider $\mathbf{c}_6=(\mathbf{c}_5,\lambda^5,1)\in\mathcal{C}$ which has weight 6 and $\mathbf{c}_7=(\mathbf{c}_5,\lambda^5,\lambda^6)$ which has weight 7. Thus $\mathcal{C}$ is an $[12,6]_{q^7/q}$-code with 7 non-zero weights.
\medskip

\noindent Finally, consider $n=12$, $k=10$ and $m=3$. In this case, $t=\left\lfloor\frac{km-n}{m}\right\rfloor=\left\lfloor\frac{30-12}{3}\right\rfloor=6$ and $a=n-(k-t-2)m=12-(10-6-2)3=6<t+2=8$.  Following Construction \ref{constr:casen>m}, we have $\beta=\lfloor\frac{n-k}{m-1}\rfloor=1$ and $\gamma=n-k-\beta(m-1)=0$, so we consider  $\textbf{n}=(\underbrace{3}_{\beta=1\text{ time}},\underbrace{1}_{\gamma+1},\underbrace{1,1,1,1,1,1,1,1}_{k-\beta-1=8 \text{ times }})$. Thus Construction \ref{constr:casen>m} gives the code $\mathcal{C}_{\lambda,\mathbf{n}}$ with generator matrix
\[
G_{\lambda, \mathbf{n}} = 
\begin{bmatrix}
\uu_{\lambda, 3} & {\bf 0} & {\bf 0} & {\bf 0} & {\bf 0} & {\bf 0}& {\bf 0}& {\bf 0}\\
{\bf 0} & \uu_{\lambda, 1} & {\bf 0} & {\bf 0} &{\bf 0} & {\bf 0}& {\bf 0}& {\bf 0}\\
{\bf 0} & {\bf 0} & \uu_{\lambda, 1} & {\bf 0} & {\bf 0}& {\bf 0}& {\bf 0}& {\bf 0}\\
{\bf 0} & {\bf 0} & {\bf 0} & \uu_{\lambda, 1} & {\bf 0}& {\bf 0}& {\bf 0}& {\bf 0}\\
{\bf 0} & {\bf 0} & {\bf 0}  & {\bf 0}& \uu_{\lambda, 1}& {\bf 0}& {\bf 0}& {\bf 0}\\
{\bf 0} & {\bf 0} & {\bf 0}  & {\bf 0} & {\bf 0}& \uu_{\lambda, 1}& {\bf 0}& {\bf 0}\\
{\bf 0} & {\bf 0} & {\bf 0}  & {\bf 0} & {\bf 0}& {\bf 0} & \uu_{\lambda, 1}& {\bf 0}\\
{\bf 0} & {\bf 0} & {\bf 0}  & {\bf 0} & {\bf 0}& {\bf 0}& {\bf 0}& \uu_{\lambda, 1}
\end{bmatrix}.
\]
This code has $m-h+1=3$ non-zero weights. Indeed, consider ${\mathbf{a}_1}=(0,0,0,0,0,0,0,0,0,0,0,1)$, ${\mathbf{a}_2}=(0,0,0,0,0,0,0,0,0,0,\lambda,1)$, ${\mathbf{a}_3}=(0,0,0,0,0,0,0,0,0,\lambda^2,\lambda,1)$. Then
\begin{itemize}
    \item $\mathbf{c}_1={\bf a}_1G_{\lambda, {\bf n}}=(0,0,0,0,0,0,0,0,0,0,0,\mathbf{u}_{\lambda,1})=(0,0,0,0,0,0,0,0,0,0,0,1)$ has weight 1.
    \item $\mathbf{c}_2={\bf a}_2G_{\lambda, {\bf n}}=(0,0,0,0,0,0,0,0,0,0,\lambda\mathbf{u}_{\lambda,1},\mathbf{u}_{\lambda,1})=(0,0,0,0,0,0,0,0,0,0,\lambda,1)$ has weight 2.
    \item $\mathbf{c}_3={\bf a}_3G_{\lambda, {\bf n}}=(0,0,0,0,0,0,0,0,0,\lambda^2\mathbf{u}_{\lambda,1},\lambda\mathbf{u}_{\lambda,1},\mathbf{u}_{\lambda,1})=(0,0,0,0,0,0,0,0,0,\lambda^2,\lambda,1)$ has weight 3.
\end{itemize}
Thus $\mathcal{C}$ is an $[12,10]_{q^3/q}$-code with 3 non-zero weights.
\end{example}

Now, to conclude the proof of Theorem \ref{thm:big_n}, we will show that $m-h+1$ is also an upper bound on $L_{\rk}(m,n,k,q)$.

Note that for $h=1$, i.e. $a<2^{t+2}$, we already know that $L_{\rk}(n,m,k,q)\leq m$.

To address the case where \( h > 1 \), which corresponds to \( a \geq 2^{t+2} \), we employ a combinatorial lemma in the same spirit as Lemma \ref{lem:combinatoriallemma}, together with the geometric description of rank-metric codes.

\begin{lemma} \label{lem:combinatoriallemman>m}
Let $h =\max\left\{ \floor*{\frac{a}{2^{t+1}}},1\right\}>1$. Let $S$ be a subset of $\{1, \dots, m\}$ such that $|S| \geq m-h+2$. Then there exist $t+2$ elements $s_{0}< \cdots< s_{t+1} \in S$ such that
\begin{equation} \label{eq:conditionsin>m}
s_0+(2^{j-1}-1)h+1 \leq s_j \leq 2^{j-1}h
\quad \mbox{ and }\quad s_{j} > \sum_{i=0}^{j-1} s_{i},\end{equation}
for every $1\leq j \leq t+1$.
\end{lemma}

\begin{proof}
The proof proceeds in a similar fashion to that of Lemma \ref{lem:combinatoriallemma}. Let \( s_0 := \min S \). Then, for each \( j = 1, \dots, t+1 \), the element \( s_j \) lies in the interval
\[
I_j = \left\{ s_0 + (2^{j-1} - 1)h + 1, \dots, 2^{j-1}h \right\}.
\]
Note that \( j = t+1 \) is the largest index for which \( 2^{j-1}h \leq m \) still holds, indeed
\[
2^t h \leq 2^t \floor*{\frac{a}{2^{t+1}}} \leq 2^t \frac{2m}{2^{t+1}}  = m,
\]
where the last inequality uses the fact that \( a \leq 2m \), as stated in Proposition~\ref{prop:m<a leq 2m}.
\end{proof}

\begin{theorem} \label{th:boundcasen>m}
Assume $n>m$. Let $h=\max\left\{\floor*{\frac{a}{2^{t+1}}},1\right\}$. Then
\[
L_{\rk}(m,n,k,q)\leq m-h+1.
\]
\end{theorem}

\begin{proof}
The case \( h = 1 \) is trivial. So suppose \( h > 1 \), which implies \( a \geq 2^{t+2} \). Assume, by contradiction, that
\[
L_{\rk}(m, n, k, q) \geq m - h + 2.
\]
Then, there exists an $[n,k]_{q^m/q}$ code $\C$ such that \( |\mathrm{WS}(\C)| \geq m - h + 2 \). By Lemma~\ref{lem:combinatoriallemman>m}, there exist \( t+2 \) distinct codewords \( \mathbf{c}_0, \ldots, \mathbf{c}_{t+1} \) such that, setting \( s_i = \w(\mathbf{c}_i) \) for \( i = 0, \ldots, t+1 \), the inequalities in \eqref{eq:conditionsin>m} are satisfied.

Moreover, since \( s_j > \sum_{i=0}^{j-1} s_i \) for every \( j = 1, \ldots, t+1 \), by using the same argument of the proof of Theorem \ref{th:upperboundn<m}, we have that the codewords \( \mathbf{c}_0, \ldots, \mathbf{c}_{t+1} \) must be \( \F_{q^m} \)-linearly independent. Let \( G \in \F_{q^m}^{k \times n} \) be a generator matrix of \( \C \), and let \( \mathbf{x}_0, \ldots, \mathbf{x}_{t+1} \in \F_{q^m}^k \) be such that \( \mathbf{c}_i = \mathbf{x}_i G \) for all \( i=1,\dots, t+1 \). \\
Let \( U \) be the $[n,k]_{q^m/q}$ system associated with \( \C \), and \( U^{\perp'} \) its geometric dual. Recall that
\[
\dim_{\F_q}(U^{\perp'}) = km - \dim_{\F_q}(U).
\]
By Proposition~\ref{prop:characweightgeometricdual}, we have
\[
\dim_{\F_q}(U^{\perp'} \cap \langle \mathbf{x}_i \rangle_{\F_{q^m}}) = m - s_i, \quad \text{for all } i = 0, \ldots, t+1.
\]
The codewords \( \mathbf{c}_0, \ldots, \mathbf{c}_{t+1} \) are \( \F_{q^m} \)-linearly independent; hence, the vectors \( \mathbf{x}_0, \ldots, \mathbf{x}_{t+1} \) are also \( \F_{q^m} \)-linearly independent. Therefore, for each \( j \in \{0, \ldots, t+1\} \), we have
\[
(U^{\perp'} \cap \langle \mathbf{x}_j \rangle_{\F_{q^m}}) \cap \bigoplus_{\substack{i=0 \\ i \neq j}}^{t+1} \left(U^{\perp'} \cap \langle \mathbf{x}_i \rangle_{\F_{q^m}} \right) = \{\mathbf{0}\}.
\]
This implies that
\[
\sum_{i=0}^{t+1} \dim_{\F_q}\left(U^{\perp'} \cap \langle \mathbf{x}_i \rangle_{\F_{q^m}} \right) \leq \dim_{\F_q}(U^{\perp'}) = km - n.
\]
Using the inequalities from \eqref{eq:conditionsin>m} and the fact that \( h \geq s_1 > s_0 \), we obtain $m-s_0\geq m-h+1$ and $m-s_i \geq m-2^{i-1}h$, for every $1 \leq i \leq t+1$. As a consequence, we get
\[
m - h + 1 + \sum_{i=0}^{t}(m - 2^i h) \leq km - n,
\]
which simplifies to
\[
(t+2)m + 1 - 2^{t+1}h \leq km - n,
\]
and hence
\[
n - (k - t - 2)m \leq 2^{t+1} h - 1,
\]
i.e.
\[
a \leq 2^{t+1} h - 1 = 2^{t+1} \left\lfloor \frac{a}{2^{t+1}} \right\rfloor - 1 \leq a - 1,
\]
which is a contradiction. The claim follows.
\end{proof}

Theorem~\ref{thm:big_n} now follows from \Cref{thm:constrn>m} and \Cref{th:boundcasen>m}. In addition, we also obtain a complete characterization of the values of \( n, m, k \) for which FWS codes exist in the case \( n > m \).

\begin{proposition} \label{prop:characterFWSn>m}
Let \( n > m \), and let $t = \floor*{\frac{km-n}{m}}$ and \( a = n - (k - t - 2)m \). A nondegenerate FWS \( [n, k]_{q^m/q} \) code exists if and only if \( a < 2^{t+2} \). In this case, the code \( \mathcal{C}_{\lambda, \mathbf{n}} \) defined as in Construction \ref{constr:casen>m} is an FWS code.
\end{proposition}

We now present an illustrative example of an FWS $\nkqm$ code for \( n > m \), that is, a code whose weight spectrum is precisely \( \{1, \ldots, m\} \).

\begin{example}
Let $n=10$, $m = 6$, and $k=4$. Then we have $\mu = km-n = 24-10 = 14$, $t = \floor*{\mu/m} = 2$, and $a = n - (k-t-2)m = 10<2^{t+2}=16$. By Proposition \ref{prop:characterFWSn>m}, we know that a full weight spectrum (FWS) code exists, and Construction~\ref{constr:casen>m} provides an explicit construction. In this case, we have $z=0$ and Construction \ref{constr:casen>m} gives the code $\C_{\lambda,\mathbf{n}}$ with $\mathbf{n}=(m_1,m_2,m_3,m_4)=(5,3,1,1)$ and generator matrix \[G_{\lambda,\mathbf{n}}=\begin{bmatrix}
\uu_{\lambda, 5} & 0 & 0 & 0 \\
0 & \uu_{\lambda, 3} & 0 & 0 \\
0 & 0 & \uu_{\lambda, 1} & 0 \\
0 & 0 & 0 & \uu_{\lambda, 1} \\
\end{bmatrix}.\] This code has codewords of every weight from $1$ to $6$. Therefore, $\C_{\lambda,\mathbf{n}}$ is an FWS $[10,4]_{q^6/q}$ code.\\
Now, consider the case $n=7$, $k=6$ and $m=3$. By Proposition \ref{prop:characterFWSn>m}, we know that a full weight spectrum (FWS) code exists and Construction \ref{thm:constrn>m} provides an explicit construction. In this case, $\mu=km-n=11$, $t=\lfloor\frac{\mu}{m}\rfloor=3$ and $a=n-(k-t-2)m=4<t+2$. Following Construction \ref{constr:casen>m}, we have $\beta=\lfloor\frac{n-k}{m-1}\rfloor=0$ and $\gamma=n-k-\beta(m-1)=1$, so we consider $\textbf{n}=(\underbrace{2}_{\gamma+1},\underbrace{1,1,1,1,1}_{k-\beta-1\, \text{times}})$. Thus Construction \ref{thm:constrn>m} gives the code $\mathcal{C}_{\lambda,\textbf{n}}$ with generator matrix
\[
G_{\lambda, \mathbf{n}} = 
\begin{bmatrix}
\uu_{\lambda, 2} & {\bf 0} & {\bf 0} & {\bf 0} & {\bf 0}& {\bf 0} \\
{\bf 0} & \uu_{\lambda, 1} & {\bf 0} & {\bf 0} &{\bf 0}& {\bf 0} \\
{\bf 0} & {\bf 0} & \uu_{\lambda, 1} & {\bf 0} & {\bf 0}& {\bf 0}\\
{\bf 0} & {\bf 0} & {\bf 0} & \uu_{\lambda, 1} & {\bf 0}& {\bf 0}\\
{\bf 0} & {\bf 0} & {\bf 0}  & {\bf 0}& \uu_{\lambda, 1}& {\bf 0}\\
{\bf 0} & {\bf 0} & {\bf 0}  & {\bf 0}& {\bf 0} & \uu_{\lambda, 1}
\end{bmatrix}.
\]
This code has $m-h+1=3$ non-zero weights. Thus, $\mathcal{C}_{\lambda,\textbf{n}}$ is an FWS $[7,6]_{q^3/q}$-code.

$\hfill \lozenge$
\end{example}

In summary, in the regime \( n \leq m \), by \Cref{thm:small_n}, we know that the function \( L_{\mathrm{rk}}(n, m, k, q) \) grows linearly with \( n \). However, for \( n > m \), \Cref{thm:big_n} establishes that its behavior becomes significantly more intricate. In particular, the function does not decrease uniformly as \( n \) increases; rather, its values depend heavily on the ratio \( n/m \). As \( n \) approaches the extremal bound \( km \), the maximum number of distinct nonzero weights tends to 1. In this case, the corresponding codes are rank-metric simplex codes. Also, note that the value of the function $L_{\mathrm{rk}}(n, m, k, q) $ does not depend on $q$.

To illustrate these behaviors, we plot the function \( L_{\mathrm{rk}}(n, m, k, q) \) for selected values of \( n, m, k, q \) as illustrative examples.

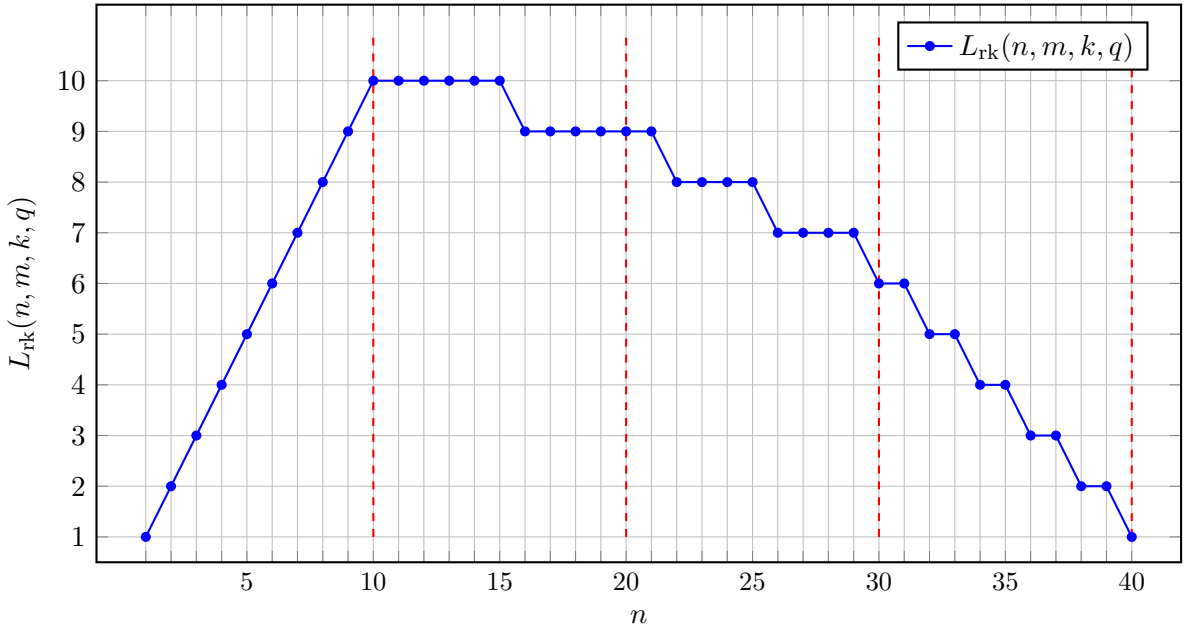
\pgfmathsetmacro{\m}{7}
\pgfmathsetmacro{\k}{3}

\newcommand{\coords}{}
\foreach \n in {1,...,21} {
  \pgfmathtruncatemacro{\cond}{\n <= \m}
  \ifnum\cond=1
    \pgfmathtruncatemacro{\s}{max(floor(\n / (2^(\k - 1))), 1)}
    \pgfmathtruncatemacro{\Lval}{\n - \s + 1}
  \else
    \pgfmathtruncatemacro{\t}{floor(\k - \n / \m)}
    \pgfmathtruncatemacro{\a}{\n - (\k - \t - 2)*\m}
    \pgfmathtruncatemacro{\s}{max(floor(\a / (2^(\t + 1))), 1)}
    \pgfmathtruncatemacro{\Lval}{\m - \s + 1}
  \fi
  \xdef\coords{\coords (\n,\Lval)}
}

\begin{center}
\begin{tikzpicture} 
  \begin{axis}[
    width=16cm,
    height=9cm,
    xlabel={$n$},
    ylabel={$L_{\rk}(n, m, k,q)$},
    grid=major,
    thick,
    legend pos=north east,
    title={Plot of $L_{\rk}(n, m, k,q)$ for $m=7$, $k=3$},
    xtick={1,...,21},
    ytick={0,...,7},
    enlargelimits=0.05
  ]

  \addplot[blue, mark=*, mark size=1.5 pt, thick] coordinates {\coords};
  \addlegendentry{$L_{\rk}(n,m,k,q)$}

  \foreach \x in {7,14,21} {
    \addplot[dashed, red, thick] coordinates {(\x,1) (\x,8)};
  }

  \end{axis}
\end{tikzpicture}
\end{center}

\pgfmathsetmacro{\m}{10}
\pgfmathsetmacro{\k}{4}

\newcommand{\coordss}{}
\foreach \n in {1,...,40} {
  \pgfmathtruncatemacro{\cond}{\n <= \m}
  \ifnum\cond=1
    \pgfmathtruncatemacro{\s}{max(floor(\n / (2^(\k - 1))), 1)}
    \pgfmathtruncatemacro{\Lval}{\n - \s + 1}
  \else
    \pgfmathtruncatemacro{\t}{floor(\k - \n / \m)}
    \pgfmathtruncatemacro{\a}{\n - (\k - \t - 2)*\m}
    \pgfmathtruncatemacro{\s}{max(floor(\a / (2^(\t + 1))), 1)}
    \pgfmathtruncatemacro{\Lval}{\m - \s + 1}
  \fi
  \xdef\coordss{\coordss (\n,\Lval)}
}

\pgfplotsset{
  my xticklabels/.style={
    xtick={1,...,40},
    xticklabels={
      , , , , 5, , , , ,10,
      , , , ,15, , , , , 20,
      , , , ,25, , , , , 30,
      , , , ,35, , , , , 40
    }
  }
}

\begin{center}
\begin{tikzpicture}  
  \begin{axis}[
    width=16cm,
    height=9cm,
    xlabel={$n$},
    ylabel={$L_{\rk}(n, m, k,q)$},
    grid=major,
    thick,
    legend pos=north east,
    title={Plot of $L_{\rk}(n, m, k,q)$ for $m=10$, $k=4$},
    enlargelimits=0.05,
    xticklabel style={font=\small},
    my xticklabels,
  ytick={1,...,10},
  yticklabels={1,2,3,4,5,6,7,8,9,10}
  ]

  \addplot[blue, mark=*, mark size=1.5pt, thick] coordinates {\coordss};
  \addlegendentry{$L_{\rk}(n,m,k,q)$}

  \foreach \x in {10,20,30,40} {
    \addplot[dashed, red, thick] coordinates {(\x,1) (\x,11)};
  }

  \end{axis}
\end{tikzpicture}
\end{center}

For all admissible values of \( n, m, k, q \), we have determined the exact value of the function \( L_{\rk}(n, m, k, q) \). This leads to the following definition.

\begin{definition}
Let \( \C \) be a nondegenerate \( [n,k]_{q^m/q} \) code. We say that \( \C \) is \textbf{\( L_{\rk} \)-optimal} if
\[
\lvert \mathrm{WS}(\C) \rvert = L_{\rk}(n, m, k, q).
\]
\end{definition}

Construction \ref{constr:casenleqm} and Construction \ref{constr:casen>m} provide explicit constructions of  \( L_{\rk} \)-optimal codes for every \( n \leq m \) (in the first case) and for \( n > m \) (in the second case), by using Theorems \ref{thm:constrnleqm} and \ref{thm:constrn>m}, respectively.
 
\begin{corollary} \label{prop:existenceofLoptimal}
For any \( k<n \leq m \), the code \( \mathcal{C}_{\lambda, \mathbf{n}} \) defined as in Construction \ref{constr:casenleqm} is an \( L_{\rk} \)-optimal \( [n,k]_{q^m/q} \) code. 

For any \( n > m \) and for any $k$ with \( n \leq km \), the code \( \mathcal{C}_{\lambda, \mathbf{n}} \) defined as in Construction \ref{constr:casen>m} is an \( L_{\rk} \)-optimal \( [n,k]_{q^m/q} \) code.
\end{corollary}

As a consequence, we obtain the following existence result.

\begin{corollary}
For all admissible values of \( n, m, k, q \), with \( n \leq km \), there exists an \( L_{\rk} \)-optimal \( [n,k]_{q^m/q} \) code.
\end{corollary}

We also note that \( L_{\rk} \)-optimal codes coincide exactly with full weight spectrum rank-metric codes in the following cases.
\begin{proposition}
In the case where \( n < 2^k \) with \( n \leq m \), and in the case where \( n < 2^{t+2} \) with \( n > m \), where \( t = \left\lfloor \frac{km - n}{m} \right\rfloor \), an \( [n,k]_{q^m/q} \) code is \( L_{\rk} \)-optimal if and only if it is an FWS code.
\end{proposition}

\begin{proof}
    This follows from Propositions~\ref{prop:characterFWSn<m} and~\ref{prop:characterFWSn>m}, respectively.
\end{proof}

We conclude this section with an observation analogous to Remark~\ref{rmk:nsmallandsgeqd}. Specifically, we note that the parameter \( h \), as defined in Theorem~\ref{thm:big_n}, is greater than or equal to the minimum distance of the code.

\begin{remark}
\label{rmk:nlargeandsgeqd}
Assume that \( km \geq n > m \), and let \( h = \max\left\{ \left\lfloor \frac{a}{2^{t+1}} \right\rfloor, 1 \right\} \), where \( a = n - (k - t - 2)m \) and \( t = \left\lfloor k - \frac{n}{m} \right\rfloor \). 

We claim that \( h \geq d \), where \( d \) is the minimum distance of any nondegenerate $L_{\rk}$-optimal \( [n, k, d]_{q^m/q} \) code. By Theorem~\ref{thm:big_n}, we have 
\[
L_{\rk}(n, m, k, q) = m - h + 1 \leq m.
\]
Since \( d \) is the minimum distance, the nonzero weights of any code must lie in the range \( \{d, d+1, \dots, m\} \), so the total number of distinct nonzero weights is at most \( m - d + 1 \). But since \( L_{\rk}(n, m, k, q) = m - h + 1 \), it follows that
\[
m - h + 1 \leq m - d + 1,
\]
which implies \( h \geq d \), as claimed.
\end{remark}

\section{On the equivalence issue}
\label{section:nonequivalence}

Since our results determine the maximum number of nonzero weights that an \( \F_{q^m} \)-linear rank-metric code can attain, a natural question arises: do there exist nonequivalent codes with the same parameters that achieve this maximum?

It is very likely that, by selecting a different primitive element \( \lambda \) in Construction~\ref{constr:casen>m} or Construction~\ref{constr:casenleqm}, one can obtain nonequivalent \( L_{\rk} \)-optimal codes. However, a detailed investigation of this possibility lies beyond the scope of the present work and is left as a direction for future research.

\bigskip

Another way to obtain nonequivalent $L_{\rk}(n,m,k,q)$-optimal codes is by choosing different values for the $n_{i}'s$ in our constructions. This might not always be possible, but we provide one example to illustrate the point.

\begin{example}
Let $n = 9$ and $m=6$. Construction \ref{constr:casen>m} would give the code $\C_{1}$ with generator matrix
$$G_{1} = \begin{bmatrix}
\uu_{\lambda, 5} & 0 & 0 \\
0 & \uu_{\lambda, 2} & 0 \\
0 & 0 & \uu_{\lambda, 2} \\
\end{bmatrix},$$
which provides the weight spectrum $\{2, 3, 4, 5, 6\}$. However, the code $\C_{2}$ with generator matrix
$$G_{2} = \begin{bmatrix}
\uu_{\lambda, 4} & 0 & 0 \\
0 & \uu_{\lambda, 3} & 0 \\
0 & 0 & \uu_{\lambda, 2} \\
\end{bmatrix}$$
has the same weight spectrum, and is not equivalent to $\C_{2}$ (since it does not have $2$ linearly independent codewords with rank $2$).

Note also that the code $\mathcal{C}_{3}$ with generator matrix
$$G_{3} = \begin{bmatrix}
\uu_{\lambda, 5} & 0 & 0 \\
0 & \uu_{\lambda, 3} & 0 \\
0 & 0 & \uu_{\lambda, 1} \\
\end{bmatrix}$$
has weight spectrum $\{ 1, 3, 4, 5, 6 \}$, which is a different from that of $\mathcal{C}_{1}$ but also reaches the maximum size. This shows that the maximum size of a weight spectrum determined by the function $L_{\rk}$ is not necessarily reached by a unique weight spectrum.

$\hfill \lozenge$
\end{example}

Thus, the example above provides non-equivalent constructions of 3-dimensional 
\(L_{\mathrm{rk}}\)-optimal codes. Indeed, although these constructions have the 
same weight spectrum, they have different weight distributions. It is reasonable 
to expect that this construction can be extended to values of \(k > 3\). In the next 
subsection, we present a classification result for 2-dimensional codes under certain 
metric assumptions.

\subsection{A classification result for $k=2$} \label{sec:classification2}

In this section, we present a classification result of two-dimensional non-degenerate codes $\C\subseteq\F_{q^m}^{n}$ having the maximum possible number of distinct non-zero weights, in the case in which $n$ is even and $2\leq n\leq m$ and $\C$ admits at least three codewords, pairwise $\Fqm$-linearly independent, of weight $\frac{n}{2}$. \\

We observe that if $k=2$, $n$ is even and $2\leq n\leq m$, then 
\[
L_{\rk}(n,m,2,q)=n-\frac{n}{2}+1=\frac{n}{2}+1.
\]
Thereafter, in Theorem \ref{thm:classrankmetricFWS} we classify, up to equivalence, the $L_{\rk}$-optimal $[n,2]_{q^m/q}$ codes $\C$ that admit at least three codewords, pairwise $\Fqm$-linearly independent, of weight $\frac{n}{2}$. This result is based on \cite[Theorem 1.2]{castello2024full} which provides a classification result of \emph{full weight spectrum one-orbit cyclic subspace codes}.\\

We start with a preliminary proposition, which guarantees that, under our assumptions, a system associated to $\C$ is, up to equivalence, of the form $S\times S$, for some $\Fq$-subspace $S$ of $\Fqm$.

\begin{proposition}
    \label{cor:U=SxS}
Let $W$ be an $\Fq$-subspace of $\Fqm^2$ such that $\dim_{\Fq}(W)=2\ell$. Assume there exist three elements ${\bf y}_1, {\bf y}_2, {\bf y}_2\in W$ such that no two of them are $\Fqm$-proportional and $\dim_{\Fq}(W\cap \langle {\bf y}_i\rangle_{\Fqm})=\ell$ for $i=1,2,3$. Then $W$ is $\mathrm{G}\mathrm{L}(2,q^m)$-equivalent to an $\Fq$-subspace $U=S \times S$,
for some $\Fq$-subspace $S$ of $\Fqm$ with $\dim_{\Fq}(S)=\ell$.
\end{proposition}
\begin{proof}
By \cite[Proposition 3.2]{napolitano2022linearsets}, $W$ is $\mathrm{G}\mathrm{L}(2,q^m)$-equivalent to an $\Fq$-subspace $U=S \times T$,
for some $\Fq$-subspaces $S$ and $T$ of $\Fqm$ with $\dim_{\Fq}(S)=\dim_{\Fq}(T)=\ell$. Let $T=\langle a_1,\dots, a_\ell\rangle_{\Fq}$, for some $a_1,\ldots,a_{\ell} \in \F_{q^m}$. By \cite[Theorem 3.4]{napolitano2023classifications}, the set of the $\Fqm$-subspaces $\langle {\bf x}\rangle_{\Fqm}$ with $\langle {\bf x} \rangle_{\F_{q^m}} \neq \langle (0,1) \rangle_{\F_{q^m}}$ such that $\dim_{\Fq}(W\cap \langle {\bf x}\rangle_{\Fqm})=\ell$ is 
\[ \{ \langle (\xi,1)\rangle_{\Fqm} \colon \xi \in a_1^{-1}S\cap \ldots \cap a_\ell^{-1} S \} \]
and its size is $q^j$ with $j=\dim_{\Fq}(a_1^{-1}S\cap \ldots \cap a_{\ell}^{-1} S)$. In particular, since there exist three elements ${\bf y'}_1, {\bf y'}_2, {\bf y'}_2\in S\times T$ such that no two of them are $\Fqm$-proportional and $\dim_{\Fq}((S   \times T)\cap \langle {\bf y'}_i\rangle_{\Fqm})=\ell$ for $i=1,2,3$, we have that $\dim_{\Fq}(a_1^{-1}S\cap \ldots \cap a_\ell^{-1} S)\geq 1$. So there exists $\xi\in a_1^{-1}S\cap \ldots \cap a_\ell^{-1} S$ with $\xi\neq 0$. Thus, $\xi a_i \in S$ for every $i=1,\dots, \ell$, and, since $\dim_{\Fq}(S)=\dim_{\Fq}( T)$, we get $S=\xi T$. Thus, $U\begin{pmatrix}
    1 & 0 \\
    0 & \xi
\end{pmatrix}=S\times S$, and hence we may assume $U=S\times S$ and the assertion is proved.
\end{proof}

We have then the following classification result.

\begin{theorem}
\label{thm:classrankmetricFWS}
Let $\C$ be an $[2\ell,2]_{q^m/q}$ code. Suppose that $\C$ has at least three codewords of weight $\ell$ such that no two of them are $\F_{q^m}$-proportional. Then $\C$ is $L_{\rk}$-optimal if and only if, up to equivalence, $\C$ admits one of the following generator matrices:

\begin{itemize}
\item 
$G=\left[\begin{matrix}
1 & \lambda & \cdots & \lambda^{\ell-1} & 0 & 0 & \cdots & 0 \\
0 & 0 & \cdots & 0 & 1 & \lambda & \cdots & \lambda^{\ell-1}
\end{matrix}\right]$, for some  $\lambda\in\F_{q^{m}}\setminus \F_{q}$ such that \[ \ell\leq \begin{cases}
            \frac{[\Fq(\lambda)\colon\Fq]+1}{2} & \mbox{if} \,\,\, \dim_{\Fq}(\Fq(\lambda))<m,\\
            \frac{m}{2} & \mbox{if} \,\,\, \dim_{\Fq}(\Fq(\lambda))=m,
        \end{cases} \]
\item 
$G = \left[
\begin{array}{*{16}{c}}
1 & \xi & \lambda & \xi\lambda & \cdots & \lambda^{l-1} & \xi\lambda^{\,l-1} & \lambda^l & 0 & \cdots & 0 & 0 & 0 & 0 & 0 & 0 \\
0 & 0 & 0 & 0 & \cdots & 0 & 0 & 0 & 1 & \xi & \lambda & \xi\lambda & \cdots & \lambda^{\,l-1} & \xi\lambda^{\,l-1} & \lambda^l
\end{array}
\right]$
\end{itemize}
where $m$ is even, $\lambda\in\F_{q^m}\setminus \F_{q^2}$, $\ell=2l+1$, $l< \frac{[\F_{q^2}(\lambda)\colon \F_{q^2}]}2$, and $\F_{q^2}=\langle 1, \xi\rangle_{\Fq}$.
\end{theorem}
\begin{proof}
Let $U$ be a $[2\ell,2]_{q^m/q}$ system associated with $\C$. Note that $\dim_{\Fq}(U)=2\ell$. Let $G$ be a generator matrix of $\C$ such that $U$ is the $\F_q$-span of the columns of $G$. By Theorem \ref{th:connection} for every $\mathbf{x} \in \F_{q^m}^2$,
\[
\w(\mathbf{x}G)=2\ell-\dim_{\Fq}(U\cap \mathbf{x}^\perp).
\]
and let $\textbf{c}_1=\mathbf{x}_1 G$, $\textbf{c}_2=\mathbf{x}_2 G$ and $\textbf{c}_3=\mathbf{x}_3 G$ be three codewords of $\C$ of weight $\ell$ that are pairwise $\Fqm$-linearly independent. Then there exist $\mathbf{y}_1$, $\mathbf{y}_2$, $\mathbf{y}_3\in\F_{q^m}^{2}$ such that $\langle \mathbf{x}_i\rangle_{\F_{q^m}}^\perp=\langle\mathbf{y}_i\rangle_{\F_{q^m}}$ and 
\[
\ell=\w(\mathbf{x}_iG)=2\ell-\dim_{\Fq}(U\cap \langle \mathbf{x}_i\rangle_{\F_{q^m}}^\perp)=2\ell-\dim_{\Fq}(U\cap \langle \mathbf{y}_i\rangle_{\F_{q^m}}),
\]
i.e.
\[
\dim_{\Fq}(U\cap \langle \mathbf{y}_i\rangle_{\F_{q^m}})=\ell,
\]
for every $i=1,2,3$. Thus, there exist at least three elements ${\bf y}_1, {\bf y}_2, {\bf y}_2\in U$ such that no two of them are $\Fqm$-proportional and $\dim_{\Fq}(U\cap \langle {\bf y}_i\rangle_{\Fqm})=\ell$ for $i=1,2,3$. By Proposition \ref{cor:U=SxS}, up to the action of $\GL(2,q^{m})$, we may assume that $U=S\times S$, for some $\Fq$-subspace $S$ of $\F_{q^m}$ of dimension $\ell$. Again by Theorem \ref{th:connection} for every $\mathbf{x} \in \F_{q^m}^2$, we have
\begin{equation}
\label{eq:weightcodewordweightpoint}
\w(\mathbf{x}G)=2\ell-\dim_{\Fq}((S\times S)\cap \langle \mathbf{x}\rangle_{\F_{q^m}}^\perp)=2\ell-\dim_{\Fq}((S\times S)\cap \langle \mathbf{y}\rangle_{\F_{q^m}}),
\end{equation}
for some $\mathbf{y}\in\F_{q^m}^2$. Note that if ${\bf y}=(\alpha,\beta)\in\Fqm^2$, then
\[
\label{eq:weightcyclishifts}
\dim_{\Fq}((S\times S)\cap \langle \mathbf{y}\rangle_{\F_{q^m}})=\dim_{\Fq}(\alpha S\cap \beta S)
\]
and so
\[
0\leq \dim_{\Fq}((S\times S)\cap \langle \mathbf{y}\rangle_{\F_{q^m}})\leq \dim_{\Fq}(S)=\ell.
\]
Hence, the weight spectrum $\mathrm{WS}(\C)=\lbrace \ell \dots, 2\ell \rbrace$, if and only if, for every $i=\ell,\dots,2\ell$, there exists at least a codeword of weight $i$. By \eqref{eq:weightcodewordweightpoint}, the weight spectrum $\mathrm{WS}(\C)$ of the code $\C$ has size $n-\frac{n}{2}+1=\ell+1$, if and only if, for every $i=0,\dots,\ell$, there exists at least ${\bf y}\in\Fqm^2$ such that $\dim_{\Fq}((S\times S)\cap \langle {\bf y}\rangle_{\Fqm})=i$. As a consequence, by \cite[Theorem 1.2]{castello2024full}, we get that $S$ is one of the following:
\begin{itemize}
        \item [(1)]$S=\langle 1,\lambda,\ldots,\lambda^{\ell-1}\rangle_{\Fq}$ for some $\lambda \in \Fqm \setminus \Fq$, where
        \[ \ell\leq \begin{cases}
            \frac{[\Fq(\lambda)\colon\Fq]+1}{2} & \mbox{if} \,\,\, \dim_{\Fq}(\Fq(\lambda))<m,\\
            \frac{m}{2} & \mbox{if} \,\,\, \dim_{\Fq}(\Fq(\lambda))=m,
        \end{cases} \]
        \item [(2)] $S=\langle 1,\lambda,\ldots,\lambda^{l-1}\rangle_{\F_{q^2}}\oplus \lambda^l\Fq$ for some $\lambda\in\Fqm \setminus \F_{q^2}$, where $\ell=2l+1$, $m$ is even  and 
        $l< \frac{[\F_{q^2}(\lambda)\colon \F_{q^2}]}2$.
    \end{itemize}
    This implies the assertion.

\end{proof}

\section{Conclusion and open problems} \label{sec:conclusions}

In this work, we have determined the maximum number of distinct nonzero rank weights that an $\F_{q^m}$-linear rank-metric code of dimension \( k \) in \( \F_{q^m}^n \) can attain. This characterization is captured by the combinatorial function \( L_{\rk}(n, m, k, q) \), whose exact value we established for all admissible parameters under the nondegeneracy condition \( n \leq km \). In particular, we provided necessary and sufficient conditions for the existence of full weight spectrum (FWS) codes and constructed explicit families of $L_{\rk}$-optimal codes.

\medskip

Several natural directions arise for future research:

\begin{itemize}
\item Constructions~\ref{constr:casenleqm} and~\ref{constr:casen>m} yield examples of \( L_{\rk} \)-optimal codes, and these constructions depend on the choice of a generating element \( \lambda \in \F_{q^m} \) over \( \F_q \). It would therefore be interesting to investigate the question of equivalence among the codes produced by these constructions, and in particular, to determine whether different choices of \( \lambda \) give rise to inequivalent \( L_{\rk} \)-optimal codes.
Also, as shown in Section \ref{section:nonequivalence}, inequivalent $L_{\rk}$-optimal codes can exist even for fixed parameters. A compelling line of investigation would be to construct new families of $L_{\rk}$-optimal codes, distinct from those in Constructions~\ref{constr:casenleqm} and~\ref{constr:casen>m}, that are provably not equivalent to the codes presented in this paper. On the other hand, considering the classification result provided in Section \ref{sec:classification2}, 
it would be interesting to extend such a classification to 
\(L_{\mathrm{rk}}\)-optimal codes also for \(k>2\), or for \(k=2\) without imposing 
any metric assumptions on the code.
    
\item In a similar direction, one could aim to classify full weight spectrum (FWS) codes, namely those whose weight spectrum coincides with that of $\F_{q^m}^n$. Note that a related classification problem has recently been explored in the subspace metric context in~\cite{castello2024full} and \cite{shi2025new}.
    
\item Our study focused on the vector rank-metric setting, where codes are $\F_{q^m}$-linear subspaces of $\F_{q^m}^n$. A natural extension would be to investigate the analogous problem in the matrix setting, where codes are $\F_q$-linear subspaces of $\F_q^{m \times n}$. Using a similar argument to that in the proof of \cite[Proposition 2]{shi2019many}, the maximum number of distinct rank weights that an \( \F_q \)-linear code of dimension \( k \) can attain is at most \( \frac{q^k - 1}{q - 1} \). In the same spirit of Hamming metric case, we define MWS rank-metric codes as the $\F_q$-subspaces of $\F_q^{m\times n}$ such that the number of distinct rank weights of $\mathcal{C}$ is $\frac{q^k-1}{q-1}$. This bound can be achieved as follows: let $\C$ be an MWS (maximum weight spectrum) code of dimension $k$ over $\F_q$ with respect to the Hamming metric, and define the map
\[
\begin{aligned}
\phi\colon \F_q^n &\longrightarrow \F_q^{n \times n} \\
(x_1, \dots, x_n) &\longmapsto \mathrm{diag}(x_1, \dots, x_n).
\end{aligned}
\]
It is easy to verify that $\rk(\phi(c)) = \wt_H(c)$ for all $c \in \C$, so $\phi(\C)$ has $\frac{q^{k}-1}{q-1}$ different weights in the rank metric as well. However, it would be interesting to determine whether this is the only example of such codes. In analogy with the Hamming metric, one may also ask: what is the smallest value of \( n \) and $m$ for which MWS rank-metric codes exist? Therefore, a foundational study of this problem in the matrix setting, particularly for fixed values of $m$ and $n$, would be of significant interest.
\end{itemize}

\section*{Acknowledgements}

The research was partially supported by the Italian National Group for Algebraic and Geometric Structures and their Applications (GNSAGA - INdAM). The third author is supported by the ANR-21-CE39-0009 - BARRACUDA (French \emph{Agence Nationale de la Recherche}) and by Project n° 50424WM - PHC GALILEE 2024 ``Algebraic and Geometric methods in coding theory''.\\

\noindent \textbf{Conflict of interest statement.} The authors declare that there is no conflict of interest. \\
\textbf{Data availability statement.} Data sharing not applicable to this article as no datasets were generated or analysed during the current study.

\bibliographystyle{abbrv}
\bibliography{biblio}

\appendix 

\section{Further properties of $L_{\rk}$-optimal codes} \label{appendix}

In this section, we study further properties of codes attaining the maximum number 
of distinct values, or in other words, \(L_{\mathrm{rk}}\)-optimal codes. 
More precisely, we show that such codes are never maximum rank distance codes 
(unless \(k = 1\)), and we investigate their behavior under duality.

\subsection{Are \( L_{\rk} \)-optimal codes MRD?}

A natural first question is whether codes that are optimal with respect to minimum distance also yield examples of codes that are optimal with respect to the function \( L_{\rk} \).

A Singleton-like bound holds for rank-metric codes~\cite{delsarte1978bilinear}. Specifically, for an \( \nkdqm \) code, one has
\begin{equation} \label{eq:boundgen}
mk \leq \max\{m,n\}  (\min\{n,m\} - d + 1).
\end{equation}
An \( \nkdqm \) code is called a \textbf{maximum rank distance (MRD)} code if it achieves equality in~\eqref{eq:boundgen}.

In this subsection, we investigate whether an \( L_{\rk}(n,m,k,q) \)-optimal code can also be an MRD code. As expected, the answer turns out to be negative.

\begin{theorem}
Let $\C$ be a nondegenerate \( \nkdqm \) code, with \( 1 < k < n \). If $\C$ is $L_{\rk}$-optimal, then $\C$ is not an \( \mathrm{MRD} \) code.
\end{theorem}

\begin{proof}
Assume by contradiction that $\C$ is an MRD code. We divide our discussion in two cases.

\textbf{\underline{Case: \( n \leq m \).}}  By definition of MRD codes, we have:
\[
mk = m(n - d + 1),
\]
which implies
\[
k = n - d + 1.
\]
Since \( n \geq k + 1 \), it follows that \( d \neq 1 \). Moreover, by~\cite[Lemma 2.1]{lunardon2018nuclei}, the weight spectrum of an MRD code is given by
\[
\mathrm{WS}(\C) = \{d, d+1, \ldots, n\}.
\]
Now, assuming \( \C \) is \( L_{\rk} \)-optimal, we must have
\[
|\mathrm{WS}(\C)| = n - d + 1 = L_{\rk}(n,m,k,q) = n - s + 1,
\]
so that \( d = s \). Since \( d \neq 1 \), we conclude that
\[
d = s = \left\lfloor \frac{n}{2^{k-1}} \right\rfloor \geq 2.
\]
This yields:
\[
\frac{n}{2^{k-1}} \geq \left\lfloor \frac{n}{2^{k-1}} \right\rfloor = d = n - k + 1,
\]
and hence
\[
2^{k-1}(n - k + 1) \leq n.
\]
Rearranging, we get
\[
n(2^{k-1} - 1) \leq 2^{k-1}(k - 1).
\]
Since \( n \geq k + 1 \), it follows that
\[
(k + 1)(2^{k-1} - 1) \leq 2^{k-1}(k - 1),
\]
which implies
\[
2^k \leq k + 1,
\]
a contradiction, as \( 2^k > k + 1 \) for all \( k \geq 2 \). 
This shows that, for \( n \leq m \), an \( L_{\rk}(n,m,k,q) \)-optimal \( [n,k]_{q^m/q} \) code cannot be MRD. 
\\

\textbf{\underline{Case: \( n > m \).}} In this case, recall that
\[
a = n - (k - t - 2)m, \quad \text{and} \quad t = \left\lfloor \frac{km - n}{m} \right\rfloor.
\]
Since \( \C \) is an MRD code, by the Singleton-like bound, we have:
\[
km = n(m - d + 1),
\]
which yields
\begin{equation} \label{eq:rewritemindistn>m}
m - d = \frac{km - n}{n}.
\end{equation}

Since \( n \geq k + 1 \), it follows that \( d \neq 1 \). Moreover, by~\cite[Lemma 2.1]{lunardon2018nuclei}, the weight spectrum of an MRD code is:
\[
\mathrm{WS}(\C) = \{ d, d+1, \ldots, m \}.
\]
Assuming \( \C \) is also \( L_{\rk} \)-optimal, we obtain
\[
|\mathrm{WS}(\C)| = m - d + 1 = L_{\rk}(n, m, k, q) = m - h + 1,
\]
and hence \( d = h \). Since \( d \neq 1 \), we conclude that
\[
d = h = \left\lfloor \frac{a}{2^{t+1}} \right\rfloor \geq 2.
\]
\textbf{Case 1:} \( t = 0 \). Then
\begin{equation}
\label{eq:m-dgeq 2m-a/2}
\begin{aligned}
m - d &= m - \left\lfloor \frac{a}{2} \right\rfloor \\
&\geq m - \frac{a}{2} \\
&= \frac{2m - a}{2}.
\end{aligned}
\end{equation}
Since \( a = n - (k - 2)m \), we have
\[
\frac{2m - a}{2} = \frac{2m - (n - (k - 2)m)}{2} = \frac{km - n}{2}.
\]
Substituting into~\eqref{eq:rewritemindistn>m}, together with \eqref{eq:m-dgeq 2m-a/2}, we get:
\[
\frac{km - n}{n} = m - d \geq \frac{km - n}{2},
\]
which implies \( n \leq 2 \), a contradiction.\\
\textbf{Case 2:} \( t \geq 1 \). From~\eqref{eq:rewritemindistn>m}, we again have
\[
d = m - \frac{km - n}{n}, \quad \text{and} \quad d = \left\lfloor \frac{a}{2^{t+1}} \right\rfloor \leq \frac{a}{2^{t+1}}.
\]
Combining, we get
\[
m - \frac{km - n}{n} \leq \frac{a}{2^{t+1}} \leq \frac{2m}{2^{t+1}} = \frac{m}{2^t},
\]
by using Proposition \ref{prop:m<a leq 2m}, which gives \( a \leq 2m \). Therefore,
\[
2^t\left( m - \frac{km - n}{n} \right) \leq m \]
from which 
\[
m \leq \frac{2^t}{2^t - 1} \cdot \frac{km - n}{n}.
\]
Using the inequality \( \frac{2^t}{2^t - 1} < 2 \) (valid for \( t \geq 1 \)), we derive
\[
m < 2 \cdot \frac{km - n}{n} \]
that yields
\[n < \frac{2km}{m + 2} < 2k.
\]
Now consider \[ t = \left\lfloor \frac{km - n}{m} \right\rfloor > \left\lfloor \frac{km - 2k}{m} \right\rfloor \geq k - 1 ,\]
since \( n < 2k \). This contradicts the definition of \( t \), which must satisfy \( t \leq k - 2 \). 

\end{proof}

\subsection{On duality}

In this section, we investigate whether the family of \( L_{\rk} \)-optimal codes is closed under duality. Specifically, we ask whether the property of attaining the maximum number of nonzero weights is preserved under dual operation.

Recall that the \textbf{dual code} of an \( [n,k]_{q^m/q} \) code \( \C \) is defined as
\[
\C^\perp = \left\{ (d_1,\ldots,d_n) \in \F_{q^m}^n \,\middle|\, \sum_{i=1}^n c_i d_i = 0 \text{ for every } (c_1,\ldots,c_n) \in \C \right\}.
\]

When referring to the parameters \( \mu, t, a, s \), and \( h \) of the dual code \( \C^{\perp} \) of \( \C \), we will denote them by \( \mu^{\perp}, t^{\perp}, a^{\perp}, s^{\perp} \), and \( h^{\perp} \), respectively.

We divide our discussion into two cases: \( n \leq m \) and \( n > m \). We begin with the first case by introducing two auxiliary lemmas.

\begin{lemma}[see \textnormal{\cite{alfarano2022linear}, Proposition 3.2.]}] \label{lm:conditionnondegdual}
An $[n,k]_{q^m/q}$ code $\C$, with $n>k$ is nondegenerate if and only if the minimum distance of its dual $\C^{\perp}$ satisfies $d(\C^{\perp}) > 1$.
\end{lemma}

\begin{lemma} \label{lm:numbertheorylemma}
Let $n \geq 5$, and let $a,b$ be two integers such that $a+b=n$. Then $\frac{n}{2^{a-1}} < 2$ or $\frac{n}{2^{b-1}} < 2$.
\end{lemma}

\begin{proof}
Assume that both $\frac{n}{2^{a-1}} \geq 2$ and $\frac{n}{2^{b-1}} \geq 2$. Then simple calculations yield
\begin{align*}
n &\geq 2^{a-1}+2^{b-1}\\
2n &\geq 2^a + 2^b \\
n&\geq \frac{2^a + 2^b}{2}\geq 2^{(a+b)/2} \text{$\quad$ (by convexity of $f(x) = 2^{x}$)}\\
n &\geq 2^{n/2} \text{$ \quad$ (since $a+b=n$)},
\end{align*}
which yields a contradiction as soon as $n \geq 5$.
\end{proof}

\begin{theorem} \label{th:dualn<m}
Let $5 \leq n \leq m$, and let $\C$ be a nondegenerate $L_{\rk}$-optimal $\nkdqm$ code. Then $\C^{\perp}$ is either degenerate or is not an $L_{\rk}$-optimal code.
\end{theorem}

\begin{proof}
Assume that $\C^{\perp}$ is nondegenerate. We write $s = \max\left\{1,\lf n/2^{k-1} \rf \right\}$ as well as $s^{\perp} = \max\left\{1, \lf n/2^{n-k-1} \rf \right\}$. Suppose by contradiction that $\lvert \mathrm{WS}(\mathcal{C}^{\perp})\lvert=n-s^{\perp}+1$. Since both $\C$ and $\C^{\perp}$ are nondegenerate, by Lemma \ref{lm:conditionnondegdual}, we have $d=d(\mathcal{C})> 1$ and $d(\C^{\perp}) > 1$.

Since $\C$ and $\C^{\perp}$ are $L_{\rk}$-optimal, by Remark \ref{rmk:nsmallandsgeqd} we have $s\geq d(\mathcal{C})\geq 2$ and $s^{\perp}\geq d(\mathcal{C}^{\perp})\geq 2$. Therefore, $s = \lf n/2^{k-1} \rf$ and $s^{\perp} = \lf n/2^{n-k-1} \rf$. Since $n \geq 5$, according to Lemma \ref{lm:numbertheorylemma}, at least one of $s$ and $s^{\perp}$ is equal to $1$, a contradiction. Therefore, if $\C^{\perp}$ is nondegenerate, then it cannot be $L_{\rk}$-optimal.
\end{proof}

\begin{remark} \label{rem:counterexample}
Note that the statement of our theorem leaves only one interesting case, i.e. $n = 4$ and $k=2$, as whenever $n \leq 3$ either $\C$ or its dual have dimension $1$ (and the same is true if $n=4$ and $k \neq 2$).

In the special case \( n = 4 \) and \( k = 2 \), by Proposition \ref{prop:existenceofLoptimal}, we know that \( L_{\rk} \)-optimal codes always exist and have a weight spectrum of size $n-s+1=4-\lfloor\frac{4}{2}\rfloor+1=3$. Moreover, if \( \C^\perp \) is also \( L_{\rk} \)-optimal, then its weight spectrum must likewise have size $n-s^{\perp}+1=4-\lfloor\frac{4}{2}\rfloor+1=3$. Since we require both \( \C \) and \( \C^\perp \) to be nondegenerate, the weight 1 is not allowed in either spectrum by Lemma \ref{lm:conditionnondegdual}. Therefore, the only possible weight spectrum for both codes is \( \{2, 3, 4\} \).

Taking
$$G = \begin{bmatrix}
1 & \lambda & 0 & 0 \\
0 & 0 & 1 & \lambda \\
\end{bmatrix}$$
and
$$G' = \begin{bmatrix}
\lambda & -1 & 0 & 0 \\
0 & 0 & \lambda & -1 \\
\end{bmatrix},$$
it is clear that the codes $\C$ and $\C'$ with generator matrices respectively $G$ and $G'$ are the dual of each other, and have the required weight spectrum $\{2,3,4\}$. This does not depend on the choice of $q$ and $m$ (as long as $m \geq 4$). Hence, in this case $\C$ and $\C^{\perp}$ are both $L_{\rk}$-optimal codes.
\end{remark}

Now, we turn on the case $n>m$. In this case, we have the following result.

\begin{theorem}
\label{th:dualn>m}
Let $n > m$ and assume that $m \geq 4$ or that $m=3$ and $n \geq 5$.
Let $\C$ be a nondegenerate $L_{\rk}$-optimal $\nkdqm$ code. Then $\C^{\perp}$ is either degenerate or is not an $L_{\rk}$-optimal code.
\end{theorem}

\begin{proof}
Assume that $\C^{\perp}$ is nondegenerate and is an $L_{\rk}$-optimal code. Let $t = \lf (km-n)/m \rf$ and $t^{\perp}= \lf ((n-k)m-n)/m \rf$, and let $a = n-(k-t-2)m$ and $a^{\perp} = n - ((n-k) - t^{\perp} - 2)m$. Consider $h = \max\left\{1, \lf a/2^{t+1} \rf\right\}$ and $h^{\perp} = \max\left\{1, \lf a^{\perp}/2^{t^{\perp}+1}\rf\right\}$.\\
By Proposition \ref{prop:m<a leq 2m}, we have $m < a \leq 2m$ and $m < a^{\perp} \leq 2m$. Since both $\C$ and $\C^{\perp}$ are nondegenerate and $L_{\rk}$-optimal codes, by arguing as in the proof of \Cref{th:dualn<m}, we get
$$\frac{a}{2^{t+1}} \geq 2 \quad \text{and} \quad \frac{a^{\perp}}{2^{t^{\perp}+1}} \geq 2.$$

Now notice that $a = 2m-r$ as in the proof of Proposition \ref{prop:m<a leq 2m}, where $r = \mu \pmod m = -n \pmod m$.
Note that we also have $\mu^{\perp} = (n-k)m - n$ and therefore $\mu^{\perp}= -n \pmod m = r \pmod m$. This yields
$$a^{\perp} = 2m-r = a.$$

Next we determine the value of $t+t^{\perp}$:
\begin{align*}
t + t^{\perp} &= \frac{\mu - r}{m} + \frac{\mu' - r}{m} \\
&= \frac{km - n - r}{m} + \frac{(n-k)m - n - r}{m} \\
&= n - 2 \frac{n+r}{m} \\
\end{align*}

Now, since we have
$$2m-r \geq 2^{t+2} \quad \text{and} \quad 2m-r \geq 2^{t^{\perp}+2},$$
by convexity of $f(x) = 2^{x+2}$, and writing $b=(t+t^{\perp})/2$, we also have
$$2m-r \geq 2^{b + 2}.$$

Substituting our value of $t+t^{\perp}$ obtained above, and using the fact that $n \geq m+1$, we obtain
\begin{align*}
2m-r &\geq 2^{\frac{n}{2} - \frac{n+r}{m} + 2} \\
2m-r &\geq 2^{\frac{m+1}{2} - \frac{m+1+r}{m} + 2}, \\
\end{align*}
which yields
\begin{equation} \label{eq:ineq}
2m-r \geq 2^{\frac{m+1}{2} - \frac{1+r}{m} + 1}
\end{equation}

To complete our proof, we want \eqref{eq:ineq} to yield a contradiction. Using our last inequality, we get
$$2m \geq 2m-r \geq 2^{\frac{m+1}{2} - \frac{1+r}{m} + 1} \geq 2^{\frac{m+1}{2}}.$$
This yields a contradiction as soon as $m \geq 7$. Also, it can be checked by direct computation that \eqref{eq:ineq} does not hold for $m \in \{4,5,6\}$, for all possible values of $r$, also yielding a contradiction.

For $m=3$, we get a contradiction if $r \in \{0, 1\}$. In the case $m=3$ and $r=2$, using the inequality
$$2m-r \geq 2^{\frac{n}{2} - \frac{n+r}{m} + 2}$$
derived above, straightforward computations give
$$n \leq 4,$$
which yields a contradiction since $n \geq 5$. This completes our proof.
\end{proof}

In the special case $m=3$, $n=4$, there is actually a counter example.

\begin{remark}
Assume $n=4$ and $m=3$. Using the inequalities
$$2m-r \geq 2^{t+2} \quad \text{and} \quad 2m-r \geq 2^{t^{\perp}+2},$$
derived above, we get $t = t' = 0$, which yields $k = n-k = 2$.
Since both $\C$ and $\C^{\perp}$ are nondegenerate and $L_{\rk}$-optimal, in both cases the weight spectrum cannot include $1$, and must therefore be $\{2, 3\}$.

This is actually reached by the same example as in Remark \ref{rem:counterexample}, namely
$$G = \begin{bmatrix}
1 & \lambda & 0 & 0 \\
0 & 0 & 1 & \lambda \\
\end{bmatrix}$$
and
$$G' = \begin{bmatrix}
\lambda & -1 & 0 & 0 \\
0 & 0 & \lambda & -1 \\
\end{bmatrix},$$
where both $\C$ and $\C^\perp$ have weight spectrum $\{2, 3\}$ (since $m=3$), independently of the value of $q$.
\end{remark}

Theorems \ref{th:dualn<m} and \ref{th:dualn>m} imply that, with a few exceptions, the property of being $L_{\rk}$-optimal is not preserved under duality.

\bigskip

\end{document}